\documentclass[a4paper,11pt]{article}

\usepackage{tikz}
\usepackage{pgfplots}
\usepackage{amsfonts,amsmath,amsthm}
\usepackage{array}
\usepackage{float}
\usepackage{subfig}
\usepackage{multirow}
\usepackage{graphicx}
\usepackage{tabularx}
\usepackage{bigstrut}
\usepackage{a4wide}

\usepackage{bbm}

\def \E {\mathbb{E}}
\def \P {\mathbb{P}}
\def \Var {{\rm Var}\,}

\def \ee {{\rm e}}
\def \eps {\varepsilon}

\def \d {\delta}

\def \dd {{\rm d}}

\newcommand{\e} {\varepsilon}

\newcommand{\Rightarrowd}{{\;\buildrel{d}\over\Rightarrow\;}}
\newcommand{\equalD}{{\;\buildrel{d}\over= \;}}

\newtheorem{theorem}{Theorem}

\newtheorem{lemma}{Lemma}

\newtheorem{assumption}{Assumption}
\newtheorem{corollary}{Corollary}
\newtheorem{proposition}{Proposition}

\newtheorem{remark}{Remark}

\begin{document}

\title{Robust heavy-traffic approximations for service systems facing overdispersed demand}

\author{Britt W.J. Mathijsen\footnote{
Department of Mathematics and Computer Science, Eindhoven
University of Technology, P.O. Box 513, 5600 MB Eindhoven, The
Netherlands (\{b.w.j.mathijsen,a.j.e.m.janssen,j.s.h.v.leeuwaarden\}@tue.nl)} \and 
A.J.E.M. Janssen\footnotemark[1] \and 
Johan S.H. van Leeuwaarden\footnotemark[1] \and 
Bert Zwart\footnotemark[1]\ \footnote{Centrum Wiskunde \& Informatica, P.O. Box 94079,
1090 GB, Amsterdam, The Netherlands (bert.zwart@cwi.nl)}
}

\maketitle

\begin{abstract}
Arrival processes to service systems often display fluctuations that are larger than anticipated under the Poisson assumption, a phenomenon that is referred to as \textit{overdispersion}.
Motivated by this, we analyze a class of discrete stochastic models for which we derive heavy-traffic approximations that are scalable in the system size.
Subsequently, we show how this leads to novel capacity sizing rules that acknowledge the presence of overdispersion. This, in turn, leads to robust approximations for performance characteristics of systems that are of moderate size and/or may not operate in heavy traffic.
\end{abstract}

\section{Introduction}\label{intro}

One of the most prevalent assumptions in queueing theory is the assumption that arrivals of jobs occur according to a Poisson process.
Although natural and convenient from a mathematical viewpoint, the Poisson assumption often fails to be confirmed in practice. A deterministic arrival rate implies that the demand over any given period is a Poisson random variable, whose variance equals its expectation. A growing number of empirical studies shows that the variance of demand typically deviates from the mean significantly. Recent work \cite{Kim2015b,Kim2015a} reports variance being strictly less than the mean in health care settings employing appointment booking systems. This reduction of variability can be accredited to the goal of the booking system to create a more predictable arrival pattern. On the other hand, in other scenarios with no control over the arrivals, the variance can dominate the mean, see \cite{Avramidis:2004,Bassamboo2010,Bassamboo2009,Brown2005,Chen2001,Gans2003,Gurvich2010,koolejongbloed,kimwhitt,maman,Mehrotra2010,Robbins2010,Steckley2009,Zan2012}. The feature that variability is higher than one expects from the Poisson assumption is referred to as {\it overdispersion} and serves as the primary motivation for this work.

Stochastic models with the Poisson assumption have been widely applied to optimize capacity levels in service systems. 
When stochastic models, however, do not take into account overdispersion, resulting performance estimates are likely to be overoptimistic. 
The system then ends up being underprovisioned, which possibly causes severe performance problems, particularly in critical loading.

A significant part of the queueing literature has focused on extending Poisson arrival processes to more bursty arrival processes, and analyze these models using, for example, matrix-analytic models
\cite{Neuts1981,Latouche1999}. In this paper, we focus on a different cause of overdispersion in arrival processes, which is \emph{arrival rate uncertainty}.
Since model primitives, in particular the arrival rate, are typically estimated through historical data, these are prone to be subject to forecasting errors.
In the realm of Poisson processes, this inherent uncertainty can be acknowledged by viewing the arrival rate $\Lambda_n$ itself as being stochastic. The resulting doubly stochastic Poisson process, also known as Cox process (first presented in \cite{Cox1955}), implies that demand in a given interval $A_{k,n}$ follows a mixed Poisson distribution.
In this case, the expected demand per period equals $\mu_n = \E[\Lambda_n]$, while the variance is $\sigma_n^2 = \E[\Lambda_n]+\Var\Lambda_n$.
By selecting the distribution of the mixing factor $\Lambda_n$, the magnitude of overdispersion can be made arbitrarily large, and only a deterministic $\Lambda_n$ leads to a true Poisson process.

The mixed Poisson model presents a useful way to fit both the mean and variance to real data, particularly in case of overdispersion.
The mixing distribution can be estimated parametrically or non-parametrically, see \cite{koolejongbloed,maman}.
A popular parametric family is the Gamma distribution, which gives rise to an effective data fitting procedure that uses the fact that a Gamma mixed Poisson random variable follows a negative binomial distribution.
We will in this paper adopt the assumption of a Gamma-Poisson mixture as the demand process.\\

We investigate the impact of this modeling assumption within the context of a classical model in queueing theory, which is the reflected random walk. 
In particular, we consider a sequence of such random walks, indexed by $n$, with increments $A_{k,n}-s_n$, where $A_{k,n}\sim\,{\rm Pois}(\Lambda_n)$ and $s_n$ is denotes the system capacity, and we consider a regime in which the system approaches heavy traffic.
We are especially interested in the impact of overdispersion on the way performance measures scale, and how they impact capacity allocation rules. 
 
A sensible candidate capacity allocation rule is $s_n = \mu_n + \beta \sigma_n$ for some $\beta>0$, which is equivalent to the scaling
\begin{equation*}
\frac{\mu_n}{\sigma_n}\,(1-\rho_n) \to \beta, \text{ for } \qquad n\to\infty,
\end{equation*}
where $\rho_n := \mu_n/s_n$ denotes the utilization.
We will verify mathematically that this is asymptotically the appropriate choice and our methods allow to quantify the accuracy of the resulting performance formulae for finite systems. 
Studies that have adressed similar capacity allocation problems with stochastic arrival rates include \cite{Kocaga2015,maman,Whitt1999,Whitt2006}.
Of the aforementioned papers, our work best relates to \cite{maman}, in the sense that we also assess the asymptotic performance of queueing system having a stochastic arrival rate in heavy traffic.
We therefore expand the paradigm of the QED regime, in order to have it accommodate for overdispersed demand that follows from a doubly stochastic Poisson process.

The first part of our analysis relates to \cite{Sigman2011b}, in which a sequence of cyclically thinned queues, denoted by $G_n/G_n/1$ queues, is considered.
Here, $G_n$ indicates that only every $n^{th}$ point of the original point process is considered. 
In this framework, it is shown that the stationary waiting time can be characterized as the maximum of a random walk, in which the increments grow indefinitely. 
Under appropriate heavy-traffic scaling, the authors prove convergence to a Gaussian random walk, and moreover characterize the limits the stationary waiting time moments.
Our work differs with respect to \cite{Sigman2011b} in the sense that we study a discrete-time model, rather than the continuous-time $G_n/G_n/1$ queue. 
Also, the presence of the overdispersion requires us to employ an alternative scaling. 

Furthermore, our approach through Pollaczek's formula, allows us to derive estimates for performance measures in pre-limit, i.e.~large but finite-size, systems.
Mathematically, this second part of our analysis is related to previous work \cite{Janssen2015}. In particular, we use a refinement of the saddle point technique to establish our asymptotic estimates. The associated analysis is substantially more involved in the present situation, as we will explain in Section \ref{sec:robust_analysis}.\\

\noindent
\textbf{Structure of the paper}. The remainder of this paper is structured as follows. Our model is introduced in Section \ref{modelSection} together with some preliminary results.
In Section 3 we derive the classical heavy-traffic scaling limits for the queue length process in the presence of overdispersed arrivals both for the moments and the distribution itself.
Section 4 presents our main theoretic result, which provides a robust refinement to the heavy-traffic characterization of the queue length measures in pre-limit systems.
In Section 5, we describe the numerical results and demonstrate the heavy-traffic approximation.

\section{Model description and preliminaries}\label{modelSection}
We consider a sequence of discrete stochastic models, indexed by $n$, in which time is divided into periods of equal length. At the beginning of each period $k=1,2,3,...$ new demand $A_{k,n}$ arrives to the system. The demands per period $A_{1,n},A_{2,n},...$ are assumed independent and equal in distribution to some non-negative integer-valued random variable $A_n$.
The system has a service capacity $s_n\in\mathbb{N}$ per period, so that the recursion
\begin{equation}
   \label{mm1}
   Q_{k+1,n} = \max\{Q_{k,n} + A_{k,n}-s_n,0\}, \qquad k=0,1,2,... ,
   \end{equation}
with $Q_{n}(0) = 0$. For brevity, we define $\mu_n:= \E A_{n}$ and $\sigma_n^2 = \Var A_n$. The duality principle shows that this expression is equivalent to
\begin{equation}
   \label{mm2}
   Q_{k+1,n} \equalD \max_{0\leq j\leq k}\Bigl\{{\sum_{i=1}^j} (A_{i,n}-s_n)\Bigr\}, \qquad k=0,1,2,... ,
   \end{equation}
   i.e. the maximum of the first $k$ steps of a random walk with steps distributed as $A_n-s_n$. Even more so, we can characterize $Q_{n}$, the stationary queue length, as
\begin{equation}
   \label{mm3}
   Q_{n} \equalD \max_{k\geq 0}\Bigl\{{\sum_{i=1}^k} (A_{i,n}-s_n)\Bigr\}.
   \end{equation}
The behavior of $Q_{k,n}$ greatly depends on the characteristics of $A_n$ and $s_n$. First, note that $\mu_n<s_n$ is a necessary condition for the maximum to be finite and therefore for the queue to be stable. Before continuing the analysis of $Q_{n}$, we impose a set of conditions on the asymptotic properties of $s_n,\mu_n$ and $\sigma_n$.

\begin{assumption}
\label{as1}
\ \\*
\vspace{-6mm}
\begin{enumerate}
\item[{\normalfont (a)}] {\rm (Asymptotic growth)}
\begin{equation*}
\mu_n,\sigma_n \to \infty, \quad \text{\rm for } n\to\infty.
\end{equation*}
\item[{\normalfont (b)}] {\rm (Persistence of overdispersion)}
\begin{equation*}
\sigma_n^2/\mu_n \to \infty \quad \text{\rm for } n\to\infty.
\end{equation*}
\item[{\normalfont (c)}] {\rm (Heavy-traffic condition)}
The utilization $\rho_n := \mu_n/s_n \to 1$ as $n\to\infty$, while
\begin{equation}\label{mm5}
s_n = \mu_n + \beta\, \sigma_n,
\end{equation}
for some $\beta > 0$. This is equivalent to requiring
\begin{equation}\label{mm4}
(1-\rho_n)\frac{\mu_n}{\sigma_n} \to \beta, \qquad \text{as }n\to\infty.
\end{equation}
\end{enumerate}
\end{assumption}

Assumption \ref{as1} is assumed to hold throughout the remainder of this paper.

Since we are mainly interested in the system behavior in heavy traffic, it is appropriate to study the queue length process in a scaled form. Substituting $s_n$ as in Assumption \ref{as1}(c), and dividing both sides of \eqref{mm3} by $\sigma_n$, gives

\begin{equation}
\label{mm6}
\frac{Q_{n}}{\sigma_n} = \max_{k\geq 0} \Bigl\{{\sum_{i=1}^k} \Bigl(\frac{A_{i,n}-\mu_n}{\sigma_n} - \beta\Bigr)\Bigr\}.
\end{equation}

By defining $\hat{Q}_n := Q_n/\sigma_n$ and $\hat{A}_{i,n} := (A_{i,n}-\mu_n)/\sigma_n$, we see that the scaled queue length process is in distribution equal to the maximum of a random walk with i.i.d. increments distributed as $\hat{A}_n-\beta$. Besides $\E\hat{A}_n = 0$ and $\Var \hat{A}_n=1$, the scaled and centered arrival counts $\hat{A}_n$ has a few other nice properties which we turn to later in this section.

The model in \eqref{mm1} is valid for any distribution of $A_n$, also for the original case where the number of arrivals follows a Poisson distribution with fixed parameter $\lambda_n$, but Assumption \ref{as1}(b) does not hold then. Instead, we assume $A_n$ to be Poisson distributed with uncertain arrival rate rendered by the  non-negative random variable $\Lambda_n$. This $\Lambda_n$ is commonly referred to as the \emph{prior} distribution, while $A_n$ is given the name of a Poisson mixture, see \cite{Grandell1997}. Given that the moment generation function of $\Lambda_n$, denoted by $M^\Lambda_n(\cdot)$, exists, we are able to express the probability generating function (pgf) of $A_n$ through the former. Namely,
\begin{equation}
\label{mm7}
\tilde{A}_n(z) = \E [\E[ z^{A_n} | \Lambda_n ] ] = \E[ \exp(\Lambda_n(z-1))] = M^\Lambda_n(z-1).
\end{equation}
From \eqref{mm7}, we get
\begin{equation}
\label{mm8}
\mu_n = \E A_n =  \E\Lambda_n,\qquad
\sigma_n^2 = \Var A_n = \Var \Lambda_n + \E\Lambda_n,
\end{equation}
so that $\mu_n<\sigma_n^2$ if $\Lambda_n$ is non-deterministic. Assumption \ref{as1}(b) hence translates to
\[\Var \Lambda_n/\E \Lambda_n\rightarrow \infty, \qquad n\rightarrow\infty.\]
The next result relates the converging behavior of the centered and scaled $\Lambda_n$ to that of $\hat{A}_n$.

\begin{lemma}\label{gaussStep}
Let $\mu_n,\sigma_n^2\rightarrow\infty$ and $\sigma_n^2/\mu_n\rightarrow\infty$. If
\begin{equation}
\hat{\Lambda}_n := \frac{\Lambda_n-\mu_n}{\sigma_n}\Rightarrowd N(0,1), \qquad \text{\normalfont for } n\rightarrow\infty,
\end{equation}
where $N(0,1)$ denotes a standard normal variable, then $\hat{A}_n$ converges weakly to a standard normal variable as $n\rightarrow\infty$.
\end{lemma}
The proof can be found in Appendix \ref{formalSec}.

The prevalent choice for  $\Lambda_n$ is the Gamma distribution. The Gamma-Poisson mixture turns out to provide a very good fit to arrival counts observed in service systems, as was observed by \cite{koolejongbloed}. Assuming $\Lambda_n$ to be of Gamma type with scale and rate parameters $a_n$ and $1/b_n$, respectively, we get
\begin{equation}
\label{r0}
\tilde{A}_n(z) = \Bigl(\frac{1}{1+b_n(1-z)}\Bigr)^{a_n},
\end{equation}
in which we recognize the pgf of a negative binomial distribution with parameters $a_n$ and $1/(b_n+1)$, so that
\begin{equation}
 \label{t21}
 \mu_n = a_nb_n,\qquad \sigma_n^2 = a_nb_n(b_n+1).
 \end{equation}

Note that in the context of a Gamma prior, the restrictions in Assumption \ref{as1} reduce to only two rules. For completeness, we include the revised list below.

\begin{assumption}\label{as2}
\ \\*
\vspace{-6mm}
\begin{enumerate}
\item[{\normalfont (a)}] {\rm (Asymptotic regime and persistence of overdispersion)}
\begin{equation*}
a_n, b_n \to \infty, \quad \text{\rm for } n\to\infty.
\end{equation*}
\item[{\normalfont (b)}] {\rm (Heavy-traffic condition)}
Let
\begin{equation*}
s_n = a_n b_n + \beta \sqrt{a_n b_n(b_n+1)},
\end{equation*}
for some $\beta>0$, or equivalently
\begin{equation*}
(1-\rho_n)\sqrt{a_n} \to \beta, \quad \text{\rm for } n\to\infty.
\end{equation*}
\end{enumerate}
\end{assumption}

The next result follows from the fact that $\Lambda_n$ is a Gamma random variable:
\begin{corollary}\label{scaledLambdaLemma}
Let $\Lambda_n\sim\text{\normalfont Gamma}(a_n,1/b_n)$, $A_n\sim{\rm Poisson }(\Lambda_n)$ and $a_n,b_n\rightarrow \infty$. Then $\hat{A}_n$ converges weakly to a standard normal random variable as $n\rightarrow \infty$.
\end{corollary}
\begin{proof}
With Lemma \ref{gaussStep} in mind, it is sufficient to prove that $\hat{\Lambda}_n\Rightarrow N(0,1)$ for this particular choice of $\Lambda_n$.  We  do this by proving the pointwise convergence of the cf of $\hat{\Lambda}_n$ to $\exp({-} t^2/2)$, the cf of the standard normal distribution. Let $\varphi_{G}(\cdot)$ denote the characteristic function of a random variable $G$. By basic properties of the cf,
\begin{align}
\varphi_{\hat{\Lambda}_n}(t) &= \ee^{-i\mu_nt/\sigma_n}\,\varphi_{\Lambda_n}(t/\sigma_n)
= \ee^{-i\mu_nt/\sigma_n} \Bigl(1-\frac{i b_nt}{\sigma_n}\Bigr)^{-a_n}\nonumber\\
&= \exp\Bigl[ -\frac{i\mu_nt}{\sigma_n}\, - a_n\,{\rm ln}\Bigl(1-\frac{i b_nt}{\sigma_n}\Bigr)\Bigr]\nonumber\\
\label{g13d}
&= \exp\Bigl[ -\frac{i\mu_nt}{\sigma_n} -a_n\Bigl( {-}\frac{i\,b_nt}{\sigma_n} + \frac{b_n^2t^2}{2\sigma_n^2} + O( b_n^3/\sigma_n^3)\Bigr)\Bigr]  \nonumber\\
&= \exp\Bigl[ -\frac{b_n\,t^2}{2(b_n+1)} + O\left(1/\sqrt{a_n}\right)\Bigr] \rightarrow \exp\left({-} t^2/2\right),
\end{align}
for $n\rightarrow\infty$. By L\'evy's continuity theorem this implies $\hat{\Lambda}_n$ is indeed asymptotically standard normal. 
\end{proof}
 The characterization of the arrival process as a Gamma-Poisson mixture is of vital importance in later sections.

\subsection{Expressions for the stationary distribution\label{expressionsSubsec}}
Our main focus is on the stationary queue length distribution, denoted by
\[\mathbb{P}(Q_{n}=i) =\lim_{k\rightarrow\infty} \mathbb{P}(Q_{k,n}=i).\]
Denote the pgf of $Q_{n}$ by
\begin{equation}
\label{t1}
\tilde{Q}_n(w) = \sum_{i=0}^\infty \P(Q_{n}=i) w^i.
\end{equation}
To continue our analysis of $Q_{n}$, we need one more condition on $A_n$.
\begin{assumption}\label{as3}
The pgf of $A_n$, denoted by $\tilde{A}_n(w)$, exists within $|z|<r_0$, for some $r_0>1$, so that all moments of $A_n$ are finite.
\end{assumption}

We next recall two characterizations of $\tilde{Q}_n(w)$ that play prominent roles in the remainder of our analysis.
The first characterization of $\tilde{Q}_n(w)$ originates from a random walk perspective. As we see from \eqref{mm3}, the (scaled) stationary queue length is equal in distribution to the all-time maximum of a random walk with i.i.d. increments distributed as $A_n-\beta$ (or $\hat{A}_n-\beta$ in the scaled setting). Spitzer's identity, see e.g. \cite[Theorem VIII4.2]{Asmussen2003}, then gives
\begin{equation}
\label{t3}
\tilde{Q}_n(w) = \exp\Bigl\{\sum_{k=1}^\infty \frac{1}{k}\,\Big(\E\big[w^{\left(\sum_{i=1}^k \{A_{i,n}-s_n\}\right)^+}\big]-1\Big)\Bigr\},
\end{equation}
where $(x)^+ = \max\{x,0\}$. Hence,
\begin{equation}
\label{t4}
\E Q_{n} = \tilde{Q}_n'(1) = \sum_{k=1}^\infty \frac{1}{k}\E\Bigl[ {\sum_{i=1}^k} (A_{i,n} - s_n) \Bigr]^+,
\end{equation}
\begin{equation}
\label{t4a}
\Var Q_{n} = \tilde{Q}_n''(1)+Q_n'(1)-\left(\tilde{Q}_n'(1)\right)^2 = \sum_{k=1}^\infty \frac{1}{k}\E\Bigl[ \Big({\sum_{i=1}^k} (A_{i,n} - s_n) \Big)^+\Bigr]^2,
\end{equation}
\begin{align}
\label{t5}
\P(Q_{n}=0) = \tilde{Q}_n(0) &= \exp\Bigl\{{-}{\sum_{k=1}^\infty}\frac{1}{k} \P\Bigl(\sum_{i=1}^k (A_{i,n}-s_n) > 0\Bigr) \Bigr\}.
\end{align}
A second characterization follows from Pollaczek's formula, see \cite{Abate1993,Janssen2015}:
\begin{equation}
\label{t6}
\tilde{Q}_n(w) = \exp\Bigl\{ \frac{1}{2\pi i}\int_{|z|=1+\eps} {\rm ln}\Bigl(\frac{w-z}{1-z}\Bigr) \,\frac{(z^{s_n}-\tilde{A}_n(z))'}{z^{s_n}-\tilde{A}_n(z)}\, dz\Bigr\},
\end{equation}
which is analytic for $|w|<r_0$, for some $r_0>1$. Therefore, $\eps>0$ has to be chosen such that $|w|<1+\eps<r_0$. This gives
\begin{equation}
\label{t7}
\E Q_{n} = \frac{1}{2\pi i} \int_{|z|=1+\eps} \frac{1}{1-z}\,\frac{(z^{s_n}-\tilde{A}_n(z))'}{z^{s_n}-\tilde{A}_n(z)}\, \dd z,
\end{equation}
\begin{equation}
\label{t7a}
\Var Q_{n} = \frac{1}{2\pi i} \int_{|z|=1+\eps} \frac{{-}z}{(1-z)^2}\,\frac{(z^{s_n}-\tilde{A}_n(z))'}{z^{s_n}-\tilde{A}_n(z)}\, \dd z,
\end{equation}
\begin{equation}
\label{t8}
\P(Q_{n}=0) = \exp\Bigl\{ \frac{1}{2\pi i}\int_{|z|=1+\eps} {\rm ln}\Bigl(\frac{z}{z-1}\Bigr) \,\frac{(z^{s_n}-\tilde{A}_n(z))'}{z^{s_n}-\tilde{A}_n(z)}\, \dd z\Bigr\}.
\end{equation}

Pollaczek-type integrals like \eqref{t6}-\eqref{t8} first occurred in the work of Pollaczek on the classical single-server queue (see \cite{Abate1993,Cohen1982,Janssen2008} for historical accounts). These integrals are fairly straightforward to evaluate numerically and hence give rise to efficient algorithms for performance evaluation \cite{Abate1993,boon2017pollaczek}. The integrals also proved useful in establishing heavy-traffic results by asymptotic evaluation of the integrals in various heavy-traffic regimes \cite{Kingman1962,Cohen1982,Janssen2015,boon2017pollaczek2}, and in this paper we follow that approach for a heavy-traffic regime that is suitable for overdispersion.

\section{Heavy-traffic limits}

In this section we present the result on the convergence of the discrete process $\hat{Q}_{n}$ to a non-degenerate limiting process and of the associated stationary moments. The latter requires an interchange of limits. Using this asymptotic result, we derive two sets of approximations for $\E Q_n$, $\Var Q_n$ and $\P(Q_{n}=0)$, that capture the limiting behavior of $Q_{n}$. The first set provides a rather crude estimation for the first cumulants of the queue length process for any arrival process $A_{n}$ satisfying Assumption \ref{as1}. The second set, which is the subject of the next section, is derived for the specific case of a Gamma prior and is therefore expected to provide more accurate, robust approximations for the performance metrics.

We start by indicating how the asymptotic properties of the scaled arrival process give rise to a proper limiting random variable describing the stationary queue length. The asymptotic normality of $\hat{A}_{n}$ provides a link with the Gaussian random walk and nearly deterministic queues \cite{Sigman2011a,Sigman2011b}.
The main results in \cite{Sigman2011a,Sigman2011b} were obtained under the assumption that $\rho_n\sim 1-\beta/\sqrt{n}$, in which case it follows from \cite[Thm.~3]{Sigman2011b} that the rescaled stationary waiting time process converges to a reflected Gaussian random walk.

We shall also identify the Gaussian random walk as the appropriate scaling limit for our stationary system. However, since the normalized natural fluctuations of our system are given by $\mu_n/\sigma_n$ instead of $\sqrt{n}$, we assume that the load grows like $\rho_n \sim 1 - \frac{\beta}{\mu_n/\sigma_n}$. Hence, in contrast to \cite{Sigman2011a,Sigman2011b}, our systems' characteristics display larger natural fluctuations, due to the mixing factor that renders the arrivals. Yet, by matching this overdispersed demand with the appropriate hedge against variability, we again obtain Gaussian limiting behavior. This is not surprising, since we saw in Lemma \ref{gaussStep} that the increments start resembling Gaussian behavior for $n\rightarrow\infty$. The following result summarizes this.

\begin{theorem}
\label{gaussianThm}
Let $\Lambda_n$ be a non-negative random variable such that $(\Lambda_n-\mu_n)/\sigma_n$ is asymptotically standard normal, with $\mu_n$ and $\sigma_n$ as defined in \eqref{mm8}, and $\E[\Lambda_n^3]<\infty$ for all $n\in\mathbb{N}$. Then under Assumption \ref{as1}, for $n\rightarrow \infty$,
\begin{enumerate}
\item[{\rm (i)}] $\hat{Q}_{n} \Rightarrowd M_\beta$,
\item[{\rm (ii)}] $\mathbb{P}(Q_{n} = 0) \rightarrow \mathbb{P}(M_\beta=0)$,
\item[{\rm (iii)}] $\E\hat{Q}_{n} \rightarrow \E M_\beta$,
\item[{\rm (iv)}] $\Var \hat Q_n \rightarrow \Var\, M_\beta$,
\end{enumerate}
where $M_\beta$ is the all-time maximum of a random walk with i.i.d. normal increments with mean $-\beta$ and unit variance.
\end{theorem}
The proof of Theorem \ref{gaussianThm} is given in Appendix \ref{formalSec}. The following result shows that Theorem \ref{gaussianThm} also applies to Gamma mixtures, which is a direct consequence of Corollary \ref{scaledLambdaLemma}.
\begin{corollary}
Let $\Lambda_n\sim$ \normalfont{Gamma}$(a_n,b_n)$. Then under Assumption \ref{as2} the four convergence results of Theorem \ref{gaussianThm} hold true.
\end{corollary}

It follows from Theorem \ref{gaussianThm} that the scaled stationary queueing process converges under \eqref{mm4} to a reflected Gaussian random walk. Hence, the performance measures of the original system should be well approximated by the performance measures of the reflected Gaussian random walk, yielding heavy-traffic approximations.

Like our original system, the Gaussian random walk falls in the classical setting of the reflected one-dimensional random walk, whose behavior is characterized by both Spitzer's identity and Pollaczek's formula. In particular, Pollaczek's formula gives rise to contour integral expressions for performance measures that are easy to evaluate numerically, also in heavy-traffic conditions. The numerical evaluation of such integrals is considered in \cite{Abate1993}. For $\E M_\beta$ such an integral is as follows
\begin{equation}
\label{g13e}
\E M_\beta = {-}\frac{1}{\pi}\int_0^\infty {\rm Re}\Bigl[\frac{1-\phi(-z)}{z^2}\Bigr]\dd y,
\end{equation}
with $\phi(z) = \exp(-\beta\,z+\tfrac12\,z^2)$, the Laplace transform of a normal random variable with mean $-\beta$ and unit variance, and $z=x+iy$ with an appropriately chosen real part $x$. Note that this integral involves complex-valued functions with complex arguments. Similar Pollaczek-type integrals exist for $\mathbb{P}(M_\beta=0)$ and $\Var M_\beta$; see \cite{Abate1993}. The following result simply rewrites these integrals in terms of a real integral and uses the fact that the scaled queue length process mimics the maximum of the Gaussian random walk for large $n$.

\begin{corollary}\label{abateThm}
Under Assumption \ref{as1}, the leading order behavior of $\mathbb{P}(Q_{n}=0)$, $\E Q_n$ and $\Var Q_n$ as $n\to\infty$ is characterized by, respectively,
\begin{equation}
\label{h1a}
 \exp\Bigl[\frac{1}{\pi} \int_0^\infty \frac{\beta/\sqrt{2}}{\tfrac12\beta^2+t^2}\,{\rm ln}\Bigl(1-e^{-\tfrac12\beta^2-t^2}\Bigr)\dd t\Bigr],\\
\end{equation}
\begin{equation}
\label{h1}
 \frac{\sqrt{2}\sigma_n}{\pi}\int_0^\infty \frac{t^2}{\tfrac12\beta^2+t^2}\, \frac{\exp(-\tfrac12\beta^2- t^2)}{1-\exp(-\tfrac12 \beta^2 - t^2)} \dd t,\\
 \end{equation}
 \begin{equation}
\label{h1b}
\frac{\sqrt{2}\beta\sigma_n^2}{\pi}\,\int_0^\infty \frac{t^2}{(\tfrac12 \beta^2+t^2)^2}\frac{\exp(-\tfrac12\beta^2- t^2)}{1-\exp(-\tfrac12 \beta^2 - t^2)} \dd t.
\end{equation}

\end{corollary}

\begin{proof}
According to \cite[Eq.~(15)]{Abate1993},
\begin{equation*}
\label{z1}
{-}\,{\rm ln}\,[\mathbb{P}(M_\beta=0)] = c_0,\quad \E M_\beta = c_1, \quad \Var\, M_\beta = c_2,
\end{equation*}
where
\begin{equation*}
\label{z2}
c_n = \frac{(-1)^nn!}{\pi} \,{\rm Re}\Bigl[\int_0^\infty \frac{{\rm ln}\,(1-\exp(\beta\,z+\tfrac12 z^2))}{z^{n+1}} \dd y\Bigr],
\end{equation*}
in which $z={-}x+i\,y$, $y\geq 0$, and $x$ is any fixed number between 0 and $2\beta$. Take $x=\beta$, so that
\begin{equation*}
 \label{z3}
 \beta z+\tfrac12 z^2 = {-}\tfrac12\beta^2 - \tfrac12 y^2\leq 0,\quad y\geq 0.
 \end{equation*}
For $n=0$, this gives
\begin{align*}
c_0 &= \frac{1}{\pi}\,{\rm Re}\Bigl[\int_0^\infty \frac{{\rm ln}\,(1-\exp({-}\tfrac12 \beta^2-\tfrac12 y^2))}{{-}\beta+i\,y} \dd y\Bigr] \nonumber\\
&= {-}\frac{1}{\pi}\,\int_0^\infty \frac{\beta}{\beta^2+y^2}\,{\rm ln}\,(1-\exp({-}\tfrac12 \beta^2- \tfrac12 y^2)) \dd y\nonumber\\
\label{z4}
&= {-}\frac{1}{\pi}\,\int_0^\infty \frac{\beta/\sqrt{2}}{\tfrac12\beta^2+t^2}\,{\rm ln}\,(1-\exp({-}\tfrac12 \beta^2-t^2)) \dd t,
\end{align*}
where we used that
\begin{equation*}
\label{z5}
{\rm Re }\Bigl[\frac{1}{{-}\beta+i\, y}\Bigr] = \frac{{-}\beta}{\beta^2+y^2},
\end{equation*}
together with the substitution $y=t\sqrt{2}$. For $n=1,2,\ldots,$ partial integration gives
\begin{align*}
c_n &= \frac{(-1)^n n!}{\pi} \, {\rm Re}\Bigl[\int_0^\infty \frac{{\rm ln}(1-\exp(-\tfrac12\beta^2-\tfrac12 y^2))}{({-}\beta+i\,y)^{n+1}} \dd y\nonumber\\
&= \frac{(-1)^{n-1}(n-1)!}{\pi}\,{\rm Im}\Bigl[\int_0^\infty {\rm ln}(1-\exp(-\tfrac12\beta^2-\tfrac12 y^2))\dd \Bigl(\frac{1}{(-\beta+i\,y)^n}\Bigr)\Bigr]\nonumber\\
\label{z6}
&= {-}\frac{(-1)^{n-1}(n-1)!}{\pi} {\rm Im}\Bigl[ \int_0^\infty \frac{y}{(-\beta+i\,y)^n}\,\frac{\exp(-\tfrac12\beta^2-\tfrac12 y^2)}{1-\exp(-\tfrac12\beta^2-\tfrac12 y^2)}\dd y\Bigr],
\end{align*}
where we have used that
\begin{equation*}
\label{z7}
{\rm Im}\Bigl[\frac{{\rm ln}(1-\exp(-\tfrac12\beta^2-\tfrac12 y^2))}{(-\beta+i\,y)^n}\Bigr]\Bigl|_0^\infty\Bigr. = 0.
\end{equation*}
Using
\begin{equation*}
\label{z8}
\frac{1}{(-\beta+i\,y)^n} = (-1)^n\,\frac{(\beta+i\,y)^n}{(\beta^2+y^2)^n},
\end{equation*}
we then get
\begin{equation*}
\label{z9}
c_n = \frac{(n-1)!}{\pi}\,{\rm Im}\,\Bigl[\int_0^\infty \frac{y(\beta+i\,y)^n}{(\beta^2+y^2)^n}\,\frac{\exp(-\tfrac12\beta^2-\tfrac12 y^2)}{1-\exp(-\tfrac12\beta^2-\tfrac12 y^2)}\dd y\Bigr],
\end{equation*}
which after substitution of $y=t\sqrt{2}$ gives \eqref{h1} and \eqref{h1b}.
\end{proof}

\section{Robust heavy-traffic approximations}\label{sec:robust_analysis}
We shall now establish robust heavy-traffic approximations for the canonical case of Gamma-Poisson mixtures; see \eqref{r0}.

\begin{theorem}\label{saddlepointThm}
Let $a_n,b_n$ and $s_n$ be as in Assumption \ref{as2}. Then the leading order behavior of $\E Q_n$ is given by
\begin{equation}
\label{r1}
\frac{\sqrt{2}\,\beta_n}{\pi}\Bigl(\frac{b_n+\rho_n}{1-\rho_n}\Bigr)\,\int_{0}^\infty \frac{t^2}{\tfrac12\beta^2_n+t^2}\,\frac{\exp({-}\tfrac12\beta^2_n-t^2)}{1-\exp({-}\tfrac12\beta^2_n-t^2)} \dd t\,(1+o(1)),
\end{equation}
where
\begin{equation}
\label{r2}
\beta_n^2 = s_n\Bigl(\frac{1-\rho_n}{b_n+1}\Bigr)^2\Bigl(1+\frac{b_n}{\rho_n}\Bigr).
\end{equation}
Furthermore, the leading order behavior of $\mathbb{P}(Q_{n}=0)$ and $\Var  Q_n$ is given by
\begin{equation*}
\label{r3}
\exp\Bigl[\frac{1}{\pi}\,\frac{b_n+\rho_n}{b_n+1}\,\int_0^\infty \frac{\beta_n/\sqrt{2}}{\tfrac12\beta^2_n+t^2}\,{\rm ln}\,\Bigl(1-\ee^{{-}\tfrac12\beta^2_n-t^2}\Bigr)\dd t\Bigr],
\end{equation*}
and
\begin{equation}
\label{r4}
\frac{\beta_n^3/\sqrt{2}}{\pi}\Bigl(\frac{b_n+\rho_n}{1-\rho_n}\Bigr)^2\Bigl(\frac{b_n+1}{b_n+\rho_n}+1\Bigr)\int_0^\infty \frac{t^2}{(\tfrac12 \beta_n+t^2)^2}\, \frac{\exp({-}\tfrac12\beta_n-t^2)}{1-\exp({-}\tfrac12\beta_n^2-t^2)}\dd t,
\end{equation}
respectively.
\end{theorem}

The proof of Theorem \ref{saddlepointThm} requires asymptotic evaluation of the Pollaczek-type integrals \eqref{t7}-\eqref{t8}, for which we shall use a \textit{non-standard} saddle-point method.
The saddle point method in its standard form is typically suitable for large deviation regimes, for instance excess probabilities, and it cannot be applied to asymptotically characterize other stationary measures such as the mean or mass at zero.
Indeed, in the presence of overdispersion the saddle point converges to one (as $n\to\infty$), which is a singular point of the integrand, and renders the standard saddle point method useless. Our non-standard saddle point method, originally proposed by \cite{debruijn} and also applied in \cite{Janssen2015}, aims specifically to overcome this challenge.
Subsequently, we apply the non-standard saddle-point method to turn these contour integrals into practical approximations.
In contrast to the setting of \cite{Janssen2015}, the analyticity radius tends to one in the setting with overdispersion, which is a singular point of the integrand.
For the proof of Theorem \ref{saddlepointThm}, we therefore modify the special saddle-point method developed in \cite{Janssen2015} to account for this circumstance.

\begin{proof}
Our starting point is the probability generating function of the number of arrivals per time slot, given in \eqref{r0}, which is analytic for $|z|<1+1/b_n=:r$. Under Assumption \ref{as2}, we consider $\E {Q}_n$ as given in  \eqref{t7}. We set
\begin{equation}
\label{a7}
g(z) = -{\rm ln }\,z+\frac{1}{s_n}\,{\rm ln }\big[\tilde{A}_{n}(z)\big] = -{\rm ln }\,z - \frac{a_n}{s_n}\,{\rm ln }\left(1+(1-z)b_n\right),
\end{equation}
to be considered in the entire complex plane with branch cuts $(-\infty,0]$ and $[r,\infty)$. The relevant saddle point $z_{\rm sp}$ is the unique zero $z$ of $g'(z)$ with $z\in(1,r_0)$. Since
\begin{equation}
\label{a8}
g'(z) = -\frac{1}{z} + \frac{\rho_n}{1+(1-z)b_n},
\end{equation}
this yields,
\begin{equation}
\label{a9}
1+(1-z_{\rm sp})b_n = \rho_n z_{\rm sp},\quad {\rm i.e., } \quad z_{\rm sp} = 1+\frac{1-\rho_n}{\rho_n+b_n}.
\end{equation}
We then find
\begin{equation}
\label{a10}
\E {Q}_n = \frac{s_n}{2\pi i} \int_{|z| = 1+\eps} \frac{g'(z)}{z-1}\,\frac{\exp(s_n\,g(z))}{1-\exp(s_n\,g(z))}\dd z,
\end{equation}
and take $1+\eps = z_{\rm sp}$. There are no problems with the branch cuts since we consider $\exp(s_ng(z))$ with integer $s_n$. \\

We continue as in \cite{Janssen2015} and thus we intend to substitute $z=z(v)$ in the integral in \eqref{a10}, where $z(v)$ satisfies
\begin{equation*}
\label{k1}
g(z(v)) = g(z_{\rm sp})-\tfrac12\,v^2\,g''(z_{\rm sp}) =: q(v)
\end{equation*}
on a range ${-}\tfrac12\delta_n \leq v\leq \tfrac12 \delta_n$ with $\delta_n \to 0$ as $n\to\infty$. 
Note that, this range depends on $n$, whereas these bounds $\pm \tfrac{1}{2} \delta_n$ remained bounded away from zero in \cite{Janssen2015}. 
This severely complicates the present analysis. 
We consider the approximate representation
\begin{equation}
\label{k2}
\frac{-s_n\,g''(z_{\rm sp})}{2\pi i}\int_{-\tfrac12 \delta_n}^{\tfrac12 \delta_n}\frac{v}{z(v)-1}\,\frac{\exp(s_n\,q(v))}{1-\exp(s_n\, q(v))} \dd v
\end{equation}
of $\E {Q}_n$. We have to operate here with additional care, since both the analyticity radius $r=1+1/b_n$ and the saddle point $z_{\rm sp}$ outside zero $r_0$ tend to 1 as $n\rightarrow\infty$. Specifically, proceeding under the assumptions that $(1-\rho_n)^2a_n$ is bounded while $a_n\rightarrow\infty$ and $b_n\geq 1$, see Assumption \ref{as2}, we have from \eqref{a9} that
\begin{equation}\label{a19}
z_{\rm sp}-1=\frac{1-\rho_n}{b_n+\rho_n} = \frac{1-\rho_n}{b_n} + O\Bigl(\frac{1-\rho_n}{b^2_n}\Bigr),
\end{equation}
where the $O$-term is small compared to $(1-\rho_n)/b_n$ when $b_n\rightarrow\infty$. Next, we approximate $r_0$, using that $r_0>1$ satisfies
\begin{equation*}
\label{a20}
{-}{\rm ln}\, r_0 - \frac{\rho_n}{b_n}\, {\rm ln}\,(1+(1-r_0)b_n) = 0.
\end{equation*}
Write $r_0 = 1+u/b_n$, so that we get the equation
\begin{align*}
0 &= {-}{\rm ln}\,\left(1+\frac{u}{b_n}\right) - \frac{\rho_n}{b_n}\,{\rm ln }(1-u)\nonumber \\
\label{a21}
&= {-}\frac{u}{b_n}\Bigl(1-\rho_n-\tfrac12\Bigl(\frac{1}{b_n}+\rho_n\Bigr)u-\tfrac{1}{3}\Bigl(\frac{-1}{b^2_n}+\rho_n\Bigr)u^2+\cdots\Bigr),
\end{align*}
where we have used the Taylor expansion of ${\rm ln}(1+x)$ at $x=0$. Thus we find
\begin{equation*}
\label{a22}
u=\frac{2(1-\rho_n)}{\rho_n+1/b_n}+O(u^2) = 2(1-\rho_n)+O((1-\rho_n)^2)+O\Bigl(\frac{1-\rho_n}{b_n}\Bigr),
\end{equation*}
and so,
\begin{equation*}
\label{a23}
r_0 = 1+2\,\frac{1-\rho_n}{b_n}+O\Bigl(\frac{(1-\rho_n)^2}{b_n}\Bigr) + O\Bigl(\frac{1-\rho_n}{b^2_n}\Bigr).
\end{equation*}
In \eqref{k2} we choose $\delta_n$ so large that the integral has converged within exponentially small error using $\pm\delta_n$ as integration limits, and, at the same time, so small that there is a convergent power series
\begin{equation}
\label{a26}
z(v) = z_{\rm sp}+iv+ \sum_{k=2}^\infty c_k(iv)^k, \qquad \text{for } |v| \leq \tfrac12 \delta_n.
\end{equation}
To achieve these goals, we supplement the information on $g(z)$, as given by $\eqref{a7}-\eqref{a9}$, by
\begin{equation}
\label{a27}
g''(z)=\frac{1}{z^2}+\frac{\rho_nb_n}{(1+(1-z)b_n)^2},\quad g''(1) = 1+\rho_nb_n,\quad g''(z_{\rm sp}) =\frac{1}{z_{\rm sp}^2}\Bigl(1+\frac{b_n}{\rho_n}\Bigr).
\end{equation}
Now
\begin{equation*}
\label{a36}
\exp(s_n\,q(v)) = \exp(s_n\,g(z_{\rm sp}))\exp(-\tfrac12\,s_n\,g''(z_{\rm sp})\,v^2),
\end{equation*}
and
\begin{equation*}
\label{a37} s_n\, g''(z_{\rm sp})v^2 = s_n\,b_nv^2(1+o(1)) = a_n(b_n\,v)^2(1+o(1)).
\end{equation*}
Therefore, \eqref{k2} approximates $\E {Q}_n$ with exponentially small error when we take $\tfrac12 \delta_n$ of the order $1/b_n$.

We next aim at showing that we have a power series for $z(v)$ as in \eqref{a26} that converges for $|v|\leq\tfrac12\delta_n$ with $\tfrac12\delta_n$ of the order $1/b_n$.

\begin{lemma}
Let
\begin{equation*}
\label{a38}
r_n:=\frac{1}{2\,b_n}-(z_{\rm sp} -1 ),\quad m_n:= \tfrac{2}{3}\rho_nr_n\sqrt{\frac{b_n+\rho_n^{-1}}{b_n+\rho_n}},
\end{equation*}
where we assume $r_n>0$. Then \eqref{a26} holds with real coefficients $c_k$ satisfying
\begin{equation}
\label{a39}
|c_k|\leq\frac{r_n}{m_n^k},\quad k=2,3,\ldots.
\end{equation}
\end{lemma}
\begin{proof}
We let
\begin{equation}
\label{a40}
G(z):=\frac{2(g(z)-g(z_{\rm sp}))}{g''(z_{\rm sp})(z-z_{\rm sp})^2}.
\end{equation}
Then $G(z_{\rm sp})=1$ and so we can write \eqref{k1} as
\begin{equation}
\label{a41}
F(z):=(z-z_{\rm sp})\sqrt{G(z)} = i v
\end{equation}
when $|z-z_{\rm sp}|$ is sufficiently small. Since $F(z_{\rm sp})=0$, $F'(z_{\rm sp})=1$, the B\"urmann-Lagrange inversion theorem implies validity of a power series as in \eqref{a26}, with real $c_k$ since $G(z)$ is positive and real for real $z$ close to $z_{\rm sp}$. We therefore just need to estimate the convergence radius of this series from below.

To this end, we start by showing that
\begin{equation}
\label{a42}
{\rm Re}[g''(z)] > \frac{4}{9}\,\rho_n^2\frac{b_n+\rho_n^{-1}}{b_n+\rho_n},\quad |z-z_{\rm sp}|\leq r_n.
\end{equation}
For this, we consider the representation
\begin{equation}
\label{a43}
G(z) = 2\int_{0}^1\int_0^1 \frac{g''(z_{\rm sp}+s\,t(z-z_{\rm sp}))}{g''(z_{\rm sp})} \,t\dd s\dd t.
\end{equation}
We have for $\zeta\in\mathbb{C}$ and $|\zeta-1|\leq 1/2b_{n}\leq 1/2$ from \eqref{a27} that
\begin{equation}
\label{a44}
{\rm Re}[g''(\zeta)] = {\rm Re}(1/\zeta^2) + \rho_nb_n\,{\rm Re}\Bigl[\Bigl(\frac{1}{1+(1-\zeta)b_n}\Bigr)^2\Bigr]\geq \tfrac{4}{9}(1+\rho_nb_n).
\end{equation}
To show the inequality in \eqref{a44}, it suffices to show that
\begin{equation}
\label{a45}
\min_{|\xi-1|\leq 1/2} {\rm Re}\Bigl(\frac{1}{\xi^2}\Bigr) = \frac{4}{9}.
\end{equation}
The minimum in \eqref{a45} is assumed at the boundary $|\xi-1|=1/2$, and for a boundary point $\xi$, we write
\begin{equation*}
\label{a46}
\xi= 1+\tfrac12\cos\theta+\tfrac12 i \sin\theta, \quad 0\leq \theta\leq 2\pi,
\end{equation*}
so that
\begin{equation*}
\label{a47}
{\rm Re}\Bigl(\frac{1}{\xi^2}\Bigr) = \frac{1+\cos\theta+\tfrac{1}{4}\cos 2\theta}{(\tfrac{5}{4}+\cos\theta)^2}.
\end{equation*}
Now
\begin{equation*}
\label{a48}
\frac{\dd}{d\theta} \Bigl[\frac{1+\cos\theta+\tfrac{1}{4}\cos2\theta}{(\tfrac{5}{4}+\cos\theta)^2}\Bigr] = \frac{\sin \theta\,(1-\cos \theta)}{4(\tfrac{5}{4}+\cos\theta)^3}
\end{equation*}
vanishes for $\theta=0,\pi,2\pi$, where ${\rm Re}(1/\xi^2)$ assumes the values $4/9$, 4, 4/9, respectively. This shows \eqref{a45}.

We use \eqref{a45} with $\xi = \zeta$ and with $\xi=1+(1-\zeta)b_n$, with
\begin{equation}
 \label{a49}
 \zeta = \zeta(s,t) = z_{\rm sp} + s\,t\,(z-z_{\rm sp}),\quad 0\leq s,\, t\leq 1,
 \end{equation}
 where we take $\zeta$ such that $|\zeta-1|\leq 1/2b_n$. It is easy to see from
 $1<z_{\rm sp}<1+1/2b_n$ that $|\zeta-1|\leq 1/2b_n$ holds when $|z-z_{\rm sp}|\leq r_n=1/2b_n-(z_{\rm sp}-1)$. We have, furthermore, from \eqref{a9} that $0<g''(z_{\rm sp})\leq 1+b_n/\rho_n$. Using this, together with \eqref{a44} where $\zeta$ is as in \eqref{a49}, yields
 \begin{equation*}
 \label{a50}
 {\rm Re}[G(z)] \leq \frac{4}{9}\,\frac{1+\rho_nb_n}{1+b_n/\rho_n}\,2\,\int_0^1\int_0^1 t\,\dd s\,\dd t = \tfrac{4}{9}\,\rho_n^2\,\frac{b_n+\rho_n^{-1}}{b_n+\rho_n}
 \end{equation*}
when $|z-z_{\rm sp}|\leq r_n$, and this is \eqref{a42}.
We therefore have from \eqref{a41} that
\begin{equation*}
\label{a51}
|F(z)|>r_n\cdot\frac{2}{3}\rho_n\sqrt{\frac{b_n+\rho_n^{-1}}{b_n+\rho_n}} = m_n,\quad |z-z_{\rm sp}|=r_n.
\end{equation*}
Hence, for any $v$ with $|v|\leq m_n$, there is exactly one solution $z=z(v)$ of the equation $F(z)-iv=0$ in $|z-z_{\rm sp}|\leq r_n$ by Rouch\'e's theorem. This $z(v)$ is given by
\begin{equation*}
\label{a52}
z(v) = \frac{1}{2\pi i}\,\int_{|z-z_{\rm sp}|=r_n} \frac{F'(z)\,z}{F(z)-iv}\dd z,
\end{equation*}
and depends analytically on $v$, $|v|\leq m_n$. From $|z(v)-z_{\rm sp}|\leq r_n$, we can finally bound the power series coefficients $c_k$ according to
\begin{equation*}
\label{a53}
|c_k| = \Bigl|\frac{1}{2\pi i}\int_{|iv|=m_n} \frac{z(v)-z_{\rm sp}}{(iv)^{k+1}}\dd(iv)\Bigr| \leq \frac{r_n}{m_n^k},
\end{equation*}
and this completes the proof of the lemma.
\end{proof}

\begin{remark}
We have $z_{\rm sp}-1=o(1/b_n)$, see \eqref{a19},  and so
\begin{equation*}
\label{a54}
r_n = \frac{1}{2b_n}(1+o(1)),\quad m_n = \frac{1}{3b_n}(1+o(1)),
\end{equation*}
implying that the radius of convergence of the series in \eqref{a26} is indeed of order $1/b_n$ (since we have assumed $b_n\geq 1$).
\end{remark}

We let $\delta_n=m_n$, and we write for $0\leq v\leq \tfrac12\delta_n$
\begin{equation*}
\label{a55}
\frac{v}{z(v)-1}+\frac{{-}v}{z({-}v)-1} = \frac{-2iv\,{\rm Im}(z(v))}{|z(v)-1|^2},
\end{equation*}
where we have used that all $c_k$ are real, so that $z(-v)=z(v)^*$, where $ ^*$ denotes the complex conjugate. Now from \eqref{a39} and realness of the $c_k$, we have
\begin{equation}
\label{a56}
{\rm Im}(z(v)) = v+\sum_{l=1}^\infty c_{2l+1}(-1)^l\,v^{2l+1} = v+O(v^3),
\end{equation}
and in similar fashion
\begin{equation}
\label{a57}
|z(v)-1|^2 = (z_{\rm sp}-1)^2+v^2+O((z_{\rm sp}-1)^2v^2) + O(v^4)
\end{equation}
when $0\leq v\leq \tfrac12\delta_n$. The order terms in \eqref{a56}-\eqref{a57} are negligible in leading order, and so we get for $\mu_{Q_{n}}$ via \eqref{k2} the leading order expression
\begin{equation*}
\label{a58}
\frac{{-}s_n\,g''(z_{\rm sp})}{2\pi i}\,\int_0^{\tfrac12\delta_n}\frac{{-}2iv^2}{(z_{\rm sp}-1)^2+v^2}\,\frac{\exp(s_n\,q(v))}{1-\exp(s_n\, q(v))}\dd v.
\end{equation*}
We finally approximate $q(v) = g(z_{\rm sp})-\tfrac12 g''(z_{\rm sp})v^2$.
There is a $z_1$, $1\leq z_1\leq z_{\rm sp}$ such that
\begin{equation*}
\label{a59}
g(z_{\rm sp}) = {-}\tfrac12(z_{\rm sp}-1)^2\,g''(z_1),
\end{equation*}
and, see \eqref{a19} and \eqref{a27},
\begin{equation*}
\label{a60}
g''(z_1) = g''(z_{\rm sp}) + O((1-\rho_n)b_n).
\end{equation*}
Hence
\begin{align}
s_n\,q(v) &= {-}\tfrac12 s_n\,g''(z_{\rm sp})\,[(z_{\rm sp}-1)^2+v^2]+O((1-\rho_n)b_ns_n(z_{\rm sp}-1)^2),\nonumber\\
&= {-}\tfrac12 s_n\,g''(z_{\rm sp})[(z_{\rm sp}-1)^2+v^2]+O((1-\rho_n)^2a_n),\label{a61}
\end{align}
where \eqref{a19} has been used and $a_nb_n = s_n(1+o(1))$ Therefore, the $O$-term in \eqref{a61} tends to 0 by our assumption that $(1-\rho_n)^2a_n$ is bounded. Thus, we get for $\mu_{Q_{n}}$ in leading order
\begin{equation}\label{a62}
\frac{s_n g''(z_{\rm sp})}{\pi} \int_{0}^{\tfrac12\delta_n}\frac{v^2}{(z_{\rm sp}-1)^2+v^2}\,
\frac{\exp(-\tfrac12 g''(z_{\rm sp})s_n((z_{\rm sp}-1)^2+v^2))}{1-\exp(-\tfrac12 g''(z_{\rm sp})s_n((z_{\rm sp}-1)^2+v^2))} \dd v,
\end{equation}
When we substitute $t=v\sqrt{s_n\,g''(z_{\rm sp})/2}$ and extend the integration in \eqref{a62} to all $t\geq 0$ (at the expense of an exponentially small error), we get for $\mu_{Q_{n}}$ in leading order
\begin{equation*}
\label{a63}
=\frac{1}{\pi}\,\sqrt{2\,s_n\,g''(z_{\rm sp})}\,\int_{0}^\infty \frac{t^2}{\tfrac12\beta_n^2}\,\frac{\exp({-}\tfrac12\beta^2_n-t^2)}{1-\exp({-}\tfrac12\beta^2_n-t^2)}\dd t,
\end{equation*}
where
\begin{equation*}
\label{a64}
\beta^2_n = s_n\,g''(z_{\rm sp})(z_{\rm sp}-1)^2.
\end{equation*}
Now using \eqref{a9} and \eqref{a27}, we get the result of Theorem \ref{saddlepointThm}. A separate analysis of $\beta_n$ is provided in Section \ref{convRobust}.
\end{proof}

\section{Main insights \& numerics}
Through Theorem \ref{saddlepointThm}, we can write \eqref{r1} as
\begin{equation*}
\label{ra1}
\E {Q}_n = \tilde{\sigma}_n\,\E[ M_{\beta_n}]
\end{equation*}
with
\begin{equation}
\label{ra5}
\tilde{\sigma}_n = \beta_n \Bigl(\frac{b_n+\rho_n}{1-\rho_n}\Bigr).
\end{equation}

This robust approximation for $\E {Q}_n$ is suggestive of the following two properties that extend beyond the mean system behavior, and hold at the level of approximating the queue by $\sigma_n$ times the Gaussian random walk:

\begin{itemize}
\item[\rm (i)] At the process level, the space should be normalized with $\sigma_n$, as in \eqref{mm7}. The approximation \eqref{r1} suggests that it is better to normalize with $\tilde{\sigma}_n$. Although $\tilde \sigma_n\to\sigma_n$ for $n\to\infty$, the $\tilde \sigma_n$ is expected to lead to sharper approximations for finite $n$.
\item[\rm (ii)] Again at the process level, it seems better to replace the original hedge $\beta$ by the robust hedge $\beta_n$. This thus means that the original system for finite $n$ is approximated by a Gaussian random walk with drift $-\beta_n$. Apart from this approximation being asymptotically correct for $n\to \infty$, it is also expected to approximate the behavior better for finite $n$.
\end{itemize}

\subsection{Convergence of the robust hedge\label{convRobust}}
We next examine the accuracy of the heavy-traffic approximations for $\E {Q}_n$ and $\sigma^2_Q$, following Corollary \ref{abateThm} and Theorem \ref{saddlepointThm}. We expect the robust approximation to be considerably better than the classical approximation when $\beta_n$ and $\tilde{\sigma}_n$ differ substantially from their limiting counterparts. Before substantiating this claim numerically, we present a result on the convergence rates of $\beta_n$ to $\beta$ and $\tilde{\sigma}_n$ to $\sigma_n$.

\begin{proposition}\label{gammanProp}
Let $a_n,b_n$ and $s_n$ as in Assumption \ref{as2}. Then
\begin{equation}
\label{r3a}
\beta_n^2 = \beta^2\Bigl(1 - \frac{1}{1+b_n+\sigma_n/\beta}\Bigr).
\end{equation}
\end{proposition}
\begin{proof}
From \eqref{r2}, we have
\begin{align*}
\beta_n^2 &= s_n\Bigl(\frac{1-\rho_n}{b_n+1}\Bigr)^2\Bigl(1+\frac{b_n}{\rho_n}\Bigr)= \frac{1}{s_n}\Bigl(\frac{s_n-a_nb_n}{b_n+1}\Bigr)^2\Bigl(1+\frac{s_n}{a_n}\Bigr)\nonumber\\
\label{x1}
&= \frac{1}{s_n}\frac{\beta^2\,a_nb_n(b_n+1)}{(b_n+1)^2}\Bigl(1+\frac{s_n}{a_n}\Bigr) = \beta^2\,\frac{b_n}{b_n+1}\,\Bigl(1+\frac{a_n}{s_n}\Bigr) =:\beta^2\,\bar{F}_n.
\end{align*}
Now, 
\begin{align*}
\bar{F_n} &= \frac{b_n}{b_n+1}\,\Bigl(1+\frac{a_n}{s_n}\Bigr) = \frac{b_n}{b_n+1}+\frac{1}{b_n+1}\,\frac{a_nb_n}{s_n}\nonumber\\
&= 1-\frac{1}{b_n+1}\,\Bigl(1-\frac{a_nb_n}{s_n}\Bigr) = 1-\frac{1}{b_n+1}\,\frac{\beta\,\sigma_n}{s_n}\nonumber\\
&= 1-\frac{1}{b_n+1}\,\frac{1}{1+\frac{\mu_n}{\beta\sigma_n}} = 1-\frac{1}{b_n+1+\frac{1}{\beta}\sqrt{a_nb_n(b_n+1)}},
\end{align*}
which together with $\sigma_n^2=a_nb_n(b_n+1)$ proves the proposition.
\end{proof}
Note that $\beta_n$ always approaches $\beta$ from below. Also, \eqref{r3a} shows that $b_n$ is the dominant factor in determining the rate of convergence of $\beta_n$.

\begin{proposition}\label{sigmanProp}
Let $\tilde{\sigma}_n$ as in \eqref{ra5}. Then
\begin{equation*}
\tilde{\sigma}_n = \sigma_n + b_n\beta_n + O(1).
\end{equation*}
\end{proposition}
\begin{proof}
Straightforward calculations give
\begin{align*}
\tilde{\sigma}_n &= \beta_n\,\Bigl(\frac{s_nb_n+a_nb_n}{s_n-a_nb_n}\Bigr) \nonumber\\
&= \frac{\beta_n}{\beta}\,\frac{b_n}{\sigma_n}\,(s_n+a_n)
= \frac{\beta_n}{\beta}\,\sqrt{\frac{b_n}{a_n(b_n+1)}}\left(a_n(b_n+1)+\beta\sqrt{a_nb_n(b_n+1)}\right)\nonumber\\
&= \frac{\beta_n}{\beta}\left(\sqrt{a_nb_n(b_n+1)}+\beta b_n\right) = \frac{\beta_n}{\beta}\,\sigma_n + \beta_n b_n.
\end{align*}
Applying Proposition \ref{gammanProp} together with the observation
\begin{equation*}
\sigma_n \sqrt{1 - \frac{1}{1+b_n+\sigma_n/\beta}} = \sigma_n(1 + O(1/\sqrt{a_n}b_n)) = \sigma_n + O(1)
\end{equation*}
yields the result.
\end{proof}

In Figure \ref{fig:convHedge}, we visualize the convergence speed of both parameters in case $\mu_n=n$, $\sigma_n = n^\delta$ with $\delta=0.7$ and $\beta=1$. This implies $a_n = n/(n^{2\delta}-1)$ and $b_n = n^{2\delta}-1$.

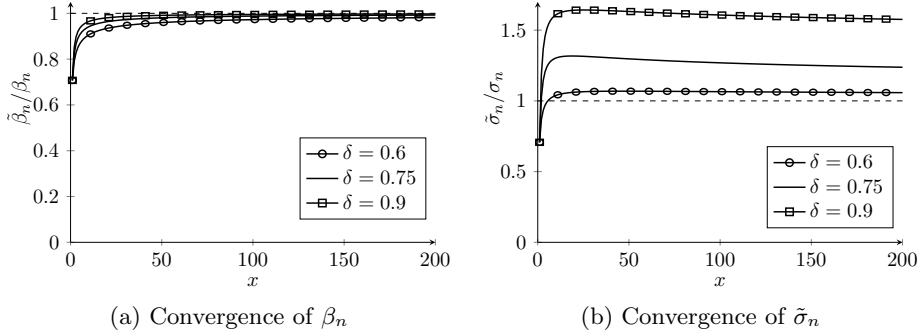
\begin{figure}
\centering
\subfloat[][Convergence of $\beta_n$]{
\centering
\begin{tikzpicture}[scale = 0.7]
\begin{axis}[
	xmin = 0,
	xmax = 200,
	ymin = 0,
	ymax = 1.05,
	xlabel = {$x$},
	ylabel = {$\tilde{\beta}_n/\beta_n$},
	y label style={at={(0.04,0.75)}},
	axis line style={->},
	axis lines = left,
	legend cell align=left,
	legend style = {at = {(axis cs: 195,0.1)},anchor = south east},
	yscale = 0.8,
	xscale = 1
]

\addplot[thick,mark = o,mark repeat = 10, mark size = 2] table[x=n,y=d06] {gamman.txt};
\addplot[thick,mark repeat = 10, mark size = 2] table[x=n,y=d075] {gamman.txt};
\addplot[thick,mark = square,mark repeat = 10, mark size = 2] table[x=n,y=d09] {gamman.txt};

\addplot[dashed] coordinates { (0,1) (200,1) };
\legend{$\delta = 0.6$, $\delta=0.75$, $\delta=0.9$};
\end{axis}
\end{tikzpicture}
}
\subfloat[][Convergence of $\tilde{\sigma}_n$]{
\centering
\begin{tikzpicture}[scale = 0.7]
\begin{axis}[
	xmin = 0,
	xmax = 200,
	ymin = 0,
	ymax = 1.7,
	xlabel = {$x$},
	ylabel = {$\tilde{\sigma}_n/\sigma_n$},
	y label style={at={(0.06,0.75)}},
	axis line style={->},
	axis lines = left,
	legend cell align=left,
	legend style = {at = {(axis cs: 195,0.1)},anchor = south east},
	yscale = 0.8,
	xscale = 1
]

\addplot[thick,mark = o, mark repeat = 10, mark size = 2] table[x=n,y=d06] {sigman.txt};
\addplot[thick,mark repeat = 10, mark size = 2] table[x=n,y=d075] {sigman.txt};
\addplot[thick,mark = square, mark repeat = 10, mark size = 2] table[x=n,y=d09] {sigman.txt};

\addplot[dashed] coordinates { (0,1) (200,1) };
\legend{$\delta = 0.6$, $\delta=0.75$, $\delta=0.9$};
\end{axis}
\end{tikzpicture}
}
\caption{Convergence of the robust hedge.}
\label{fig:convHedge}
\end{figure}
\noindent
We observe that $\beta_n$ starts resembling $\beta$ fairly quickly, as predicted by Proposition \ref{gammanProp}; $\tilde{\sigma}_n$, on the other hand, converges extremely slowly to its limiting counterpart. Since $\E{Q}_n$ and $\Var Q_n$ are approximated by $\tilde{\beta}_n$ and $\tilde{\sigma}_n$, multiplied by a term that remains almost constant as $n$ grows, the substitution of $\sigma_n$ by $\tilde{\sigma}_n$, is essential for obtaining accurate approximations, as we illustrate further in the next subsection.

\subsection{Comparison between heavy-traffic approximations}
We set $\mu_n=n$ and $\sigma^2_n=n^{2\delta}$ with $\delta>\tfrac{1}{2}$, so that  $s_n = n+\beta n^{\delta}$, and $a_n =n/(n^{2\delta-1}-1)$ and $b_n = n^{2\delta-1}-1$.
\begin{table}
\centering
\begin{tabular}{|r|r|rrr|rrr|}
\hline
$s_n$     & $\rho_n$ & $\E Q_n$ & \eqref{h1} & \eqref{r1} & $\sqrt{\Var Q_n}$ & \eqref{h1b} & \eqref{r4} \bigstrut \\
\hline
5     & 0.609 & 0.343 & 0.246 & 0.363 & 1.002 & 0.835 & 0.978 \bigstrut[t] \\
10    & 0.683 & 0.535 & 0.400 & 0.551 & 1.239 & 1.063 & 1.216 \\
50    & 0.815 & 1.405 & 1.168 & 1.405 & 1.995 & 1.817 & 1.971 \\
100   & 0.855 & 2.113 & 1.824 & 2.105 & 2.445 & 2.270 & 2.420 \\
500   & 0.920 & 5.446 & 5.006 & 5.412 & 3.923 & 3.762 & 3.899 \bigstrut[b] \\
\hline
\end{tabular}
\caption{Numerical results for the Gamma-Poisson case with $\beta=1$ and $\delta=0.6$.}
\label{gammaPoisson1}
\end{table}
\begin{table}
\centering
\begin{tabular}{|r|r|rrr|rrr|}
\hline
$s_n$     & $\rho_n$ & $\E Q_n$ & \eqref{h1} & \eqref{r1} & $\sqrt{\Var Q_n}$ & \eqref{h1b} & \eqref{r4} \bigstrut \\
\hline
5     & 0.550 & 0.462 & 0.284 & 0.479 & 1.162 & 0.896 & 1.130 \bigstrut[t]\\
10    & 0.587 & 0.852 & 0.521 & 0.855 & 1.570 & 1.213 & 1.528 \\
50    & 0.668 & 3.197 & 2.093 & 3.106 & 3.025 & 2.433 & 2.947 \\
100   & 0.700 & 5.561 & 3.784& 5.377 & 3.983 & 3.270 & 3.887\\
500   & 0.766 & 19.887 & 14.741 & 19.202 & 7.514 & 6.455 & 7.361 \bigstrut[b]\\
\hline
\end{tabular}
\caption{Numerical results for the Gamma-Poisson case with $\beta=1$ and $\delta=0.8$.}
\label{gammaPoisson2}
\end{table}
\begin{table}
\centering
\begin{tabular}{|r|r|rrr|rrr|}
\hline
$s_n$     & $\rho_n$ & $\E Q_n$ & \eqref{h1} & \eqref{r1} & $\sqrt{\Var Q_n}$ & \eqref{h1b} & \eqref{r4} \bigstrut \\
\hline
5     & 0.949 & 11.532 & 11.306 & 11.495 & 3.634 & 3.559 & 3.602 \bigstrut[t] \\
10    & 0.961 & 17.565 & 17.268 & 17.548 & 4.474& 4.398 & 4.444 \\
50    & 0.979 & 46.368 & 45.869 & 46.418 & 7.241 & 7.168 & 7.218 \\
100   & 0.984 & 70.340 & 69.735 & 70.430 & 8.910 & 8.839 & 8.888 \\
500   & 0.991 & 184.900 & 183.989 & 185.108 & 14.422 & 14.357 & 14.404 \bigstrut[b]\\
\hline
\end{tabular}
\caption{Numerical results for the Gamma-Poisson case with $\beta=0.1$ and $\delta=0.6$.}
\label{gammaPoisson3}
\end{table}
\begin{table}
\centering
\begin{tabular}{|r|r|rrr|rrr|}
\hline
$s_n$     & $\rho_n$ & $\E Q_n$ & \eqref{h1} & \eqref{r1} & $\sqrt{\Var Q_n}$ & \eqref{h1b} & \eqref{r4} \bigstrut \\
\hline
5     & 0.931 & 15.730 & 15.209 & 15.909 & 4.276 & 4.127 & 4.233 \bigstrut[t]\\
10    & 0.939 & 27.561 & 26.672 & 27.958 & 5.652 & 5.466 & 5.605 \\
50    & 0.955 & 100.660 & 97.967 & 102.070 & 10.760 & 10.476 & 10.698 \\
100   & 0.961 & 175.591 & 171.360 & 177.818 & 14.189 & 13.855 & 14.117 \\
500   & 0.971 & 638.097 & 626.346 & 644.105 & 26.963 & 26.490 & 26.864 \bigstrut[b]\\
\hline
\end{tabular}
\caption{Numerical results for the Gamma-Poisson case with $\beta=0.1$ and $\delta=0.8$.}
\label{gammaPoisson4}
\end{table}
Tables \ref{gammaPoisson1} to \ref{gammaPoisson4} present numerical results for various parameter values. The exact values are calculated using the method in Appendix \ref{numprocs}.
Several conclusions are drawn from these tables. 
Observe that the heavy-traffic approximations based on the Gaussian random walk, \eqref{h1} and \eqref{h1b}, capture the right order of magnitude for both $\E Q_n$ and $\Var Q_n$. However, the values are off, in particular for small $s_n$ and relatively low $\rho_n := \E[A_{n}] / s_n$. The inaccuracy also increases with the level of overdispersion. In contrast, the approximations that follow from Theorem \ref{saddlepointThm}, \eqref{r1} and \eqref{r4} are remarkably accurate. Even for small systems with $s_n = 5$ or 10, the approximations for $\E Q_n$ are within 6$\%$ of the exact value for small $\rho_n$ and within $2\%$ for $\rho_n$ close to 1. For $\sigma_Q^2$, these percentages even reduce to $3\%$ and $1\%$, respectively. For larger values of $s_n$ these relative errors naturally reduce further. Overall, we observe that the approximations improve for heavily loaded systems, and the corrected approximations are particularly useful for systems with increased overdispersion.

\appendix

\bibliographystyle{plain}
\bibliography{bibliography}

\section{Proofs of convergence results}
 \label{formalSec}

This section presents the details of the proof of Lemma \ref{gaussStep} and Theorem \ref{gaussianThm}, using the random walk perspective of the process $\{Q_{k,n}\}_{k=0}^\infty$. This section is structured as follows. The next two lemmata are necessary for proving the first assertion of Theorem \ref{gaussianThm}, concerning the weak convergence of the scaled process to the maximum of the Gaussian random walk, which is summarized in Proposition \ref{prop6}. The two remaining propositions of this section show convergence of $\hat{Q}_{n}$ at the process level as well as in terms of the three characteristics.

Let us first fix some notation:
\begin{equation}
\label{b1}
Y_{k,n} := \hat{A}_{k,n}-\beta,\quad
S_{k,n} = \sum_{i=1}^k Y_{i,n},
\end{equation}
with $S_{0,n} = 0$ and $k=1,2,...$. Then \eqref{mm6} can be rewritten as
\begin{equation}
\label{g5a}
 \hat{Q}_{n} \equalD \max_{0\leq k} \Bigl\{ \sum_{i=1}^k Y_{i,n}\Bigr\} =: M_{\beta,n},
\end{equation}
Last, we introduce the sequence of independent normal random variables $Z_1,Z_2,\ldots$ with mean $\-\beta$ and unit variance 1, and
\begin{equation*}
M_\beta \equalD \max_{k\geq 0} \{\sum_{i=1}^k Z_i\}
\end{equation*}

\subsection{Proof of Lemma \ref{gaussStep}}
\begin{proof}
We show weak convergence of the random variable $\hat{A}_{n}$, as defined in \eqref{b1}, to a standard normal random variable. Since $\hat{\Lambda}_n$ is asymptotically standard normal, its characteristic function converges pointwise to the corresponding limiting characteristic function, i.e.
\begin{equation}
 \label{g8}
 \lim_{n\rightarrow\infty} \phi_{\hat{\Lambda}_n}(t) = \lim_{n\rightarrow \infty} \ee^{-i\mu_nt/\sigma_n}\,\phi_{\Lambda_n}(t/\sigma_n) = \ee^{{-}t^2/2},\qquad \forall t\in \mathbb{R}.
 \end{equation}
Furthermore, by definition of $A_{n}$,
\begin{equation*}
\label{g9}
\phi_{A_{n}}(t) = \E\left[ \exp(\Lambda_n(\ee^{it}-1))\right] = \phi_{\Lambda_n}\left(-i(\ee^{it}-1)\right),
\end{equation*}
so that
\begin{equation}
\label{g10}
\phi_{\hat{A}_{k,n}}(t) = \ee^{-i\mu_nt/\sigma_n}\,\phi_{A_{k,n}}(t/\sigma_n) = \ee^{-i\mu_nt/\sigma_n}\phi_{\Lambda_n}\left(-i(\ee^{it/\sigma_n}-1)\right).
\end{equation}
Now fix $t\in\mathbb{R}$. By using
\begin{equation*}
\label{g11}
-i(\e^{it/\sigma_n} - 1) = \frac{t}{\sigma_n} -\frac{it^2}{2\sigma_n^2} + O\left(t^3/\sigma_n^3\right),
\end{equation*}
we expand the last term in \eqref{g10},
\begin{equation*}
\label{g12}
\phi_{\Lambda_n}(t/\sigma_n) + \Bigl(-\frac{i\,t^2}{2\sigma_n^2}+O\left(t^3/\sigma_n^3\right)\Bigr)
\phi_{\Lambda_n}'(t/\sigma_n) + O\Bigl(\Bigl(-\frac{i\,t^2}{2\sigma_n^2}+O\left(\frac{t^3}{\sigma_n^3}\right)\Bigr)^2\phi_{\Lambda_n}''\Big(\frac{t}{\sigma_n}\Big)\Bigr)
\end{equation*}
\begin{equation*}
\label{g13}
= \phi_{\Lambda_n}(t/\sigma_n) - \Bigl(\frac{i\,t^2}{2\sigma_n^2}+O\left(t^3/\sigma_n^3\right)\Bigr)
\phi_{\Lambda_n}'(\zeta)
\end{equation*}
for some $\zeta$ such that $|\zeta - t/\sigma_n| < |i(1-\ee^{it/\sigma_n})-t/\sigma_n|$. Also,
\begin{align}
|\phi_{\Lambda_n}'(u)| &= \left|\frac{\d}{\dd u}\int_{-\infty}^\infty \ee^{iux}\dd F_{\Lambda_n}(x)\right| = \left|\int_{0}^{\infty} ix\,\ee^{iux}\dd F_{\Lambda_n}(x)\right| \nonumber\\
\label{g13a}
&\leq \int_{-\infty}^\infty |ix\,\ee^{iux}|\,\dd F_{\Lambda_n}(x) = \int_0^\infty x\dd F_{\Lambda_n}(x) = \mu_n
\end{align}
for all $u\in\mathbb{R}$. Hence, by substituting \eqref{g10},
\begin{align}
\left| \phi_{\hat{A}_{k,n}}(t)-\ee^{-i\mu_nt/\sigma_n}\phi_{\Lambda_n}(t/\sigma_n)\right| &= \left|\ee^{-i\mu_nt/\sigma_n}\,\left(\frac{i\,t^2}{2\sigma_n^2}+O(t^3/\sigma_n^3)\right)\,\phi_{\Lambda_n}'(\zeta)\right|\nonumber\\
& \leq \left(\frac{t^2}{2\sigma_n^2}+O(t^3/\sigma_n^3)\right) |\phi_{\Lambda_n}'(\zeta)|\nonumber\\
& = \frac{\mu_n t^2}{\sigma_n^2} + O\left(\frac{\mu_nt^3}{\sigma_n^3}\right),
\label{g13b}
\end{align}
which tends to zero as $n\rightarrow \infty$ by our assumption that $\mu_n/\sigma_n^2\rightarrow 0$.
Finally,
\begin{equation*}
\label{g13c}
\left| \phi_{\hat{A}_{k,n}}(t)-\ee^{-\tfrac12 t^2}\right| \leq \left| \phi_{\hat{A}_{k,n}}(t)-\ee^{-i\mu_nt/\sigma_n}\phi_{\Lambda_n}(t/\sigma_n)\right| +
\left| \ee^{-i\mu_nt/\sigma_n}\phi_{\Lambda_n}(t/\sigma_n) - \ee^{-\tfrac12 t^2}\right|,
\end{equation*}
in which both terms go to zero for $n\rightarrow \infty$, by \eqref{g8} and \eqref{g13b}. Hence $\phi_{\hat{A}_{k,n}}(t)$ converges to $\ee^{{-}t^2/2}$ for all $t\in\mathbb{R}$, so that we can conclude by L\'evy's continuity theorem that $\hat{A}_{k,n} \Rightarrowd N(0,1)$.
\end{proof}
\subsection{Proof of Theorem \ref{gaussianThm}}
To secure convergence in distribution of $\hat{Q}_{n}$ to $M_\beta$, i.e. the maximum of a Gaussian random walk with negative drift, the first assertion of Theorem \ref{gaussianThm}.
 the following property of the sequence $\{Y_{k,n}\}_{n\in\mathbb{N}}$ needs to hold.
\begin{lemma}\label{uilemma}
Let $Y_{k,n}$ be defined as in \eqref{b1} with $\mu_n,\sigma_n^2 < \infty$ for all $n\in\mathbb{N}$. Then the sequence $\{(Y_{k,n})^+\}_{n\in\mathbb{N}}$ is uniform integrable, i.e.
\begin{equation*}
\label{g14}
\lim_{K\rightarrow\infty}\sup_n \E\Big[Y_{k,n}^+ |\mathbbm{1}_{\{|Y^{+}_{k,n}|\geq K\}}\Big] = 0.
\end{equation*}
\end{lemma}
\begin{proof}
Because the sequence $\{Y_{k,n}\}_{k\in\mathbb{N}}$ is i.i.d. for all $n$, we  omit the index $k$ in this proof. First, fix $K>0$ and note that
\begin{equation*}
\label{g15}
\E[|Y^{+}_n|\mathbbm{1}{\{|Y^{+}_n|\geq K\}}] = \E[Y^{+}_n\mathbbm{1}{\{Y^{+}_n\geq K\}}] = \E[Y_{n}\mathbbm{1}_{\{Y_{n}\geq K\}}].
\end{equation*}
This last expression can be bounded from above using the Cauchy-Schwarz inequality, so that
\begin{equation*}
\label{g16}
\E[Y_{n}\mathbbm{1}_{\{Y_{n}\geq K\}}] \leq \E[ Y^2_n]^{1/2}\,\mathbb{P}(Y_{n}\geq K)^{1/2}.
\end{equation*}
By the definition of $Y_{n}$, we know $\E [Y_{n}] = -\beta$ and $\Var Y_{n} = \Var A_{n} / \sigma_n^2 = 1$. Using this information, we find
\begin{equation*}
   \label{g17}
   \E[Y_n^2] = \Var Y_{n} + (\E[Y_{n}])^2 = 1+\beta^2
   \end{equation*}
   and
   \begin{align*}
   \mathbb{P}(Y_{n}\geq K )&=\mathbb{P}(Y_{n}+\beta\geq K+\beta) \leq \mathbb{P}(|Y_{n}+\beta|\geq K+\beta)\nonumber\\
   &\leq \frac{\Var Y_{n}}{(K+\beta)^2} = \frac{1}{(K+\beta)^2},
   \end{align*}
   where we used Chebyshev's inequality for the last upper bound. Therefore,
 \begin{align*}
\lim_{K\rightarrow \infty} \sup_n \E[|Y_n^{+}|\mathbbm{1}_{\{|Y_n^{+}|\geq K\}}] &=
\lim_{K\rightarrow \infty} \sup_n \E[Y_{n}\mathbbm{1}_{\{Y_{n}\geq K\}}]\nonumber\\
&\leq \lim_{K\rightarrow \infty} \sup_n \E[Y_n^2]^{1/2}\,\mathbb{P}(Y_{n}\geq K )^{1/2}\nonumber\\
&\leq \lim_{K\rightarrow \infty} \frac{\sqrt{1+\beta^2}}{K+\beta} = 0.
\end{align*}

\end{proof}
By combining the properties proved in Lemma \ref{gaussStep} and \ref{uilemma} with Assumption \ref{as2}, the next result follows directly by \cite[Thm.~X6.1]{Asmussen2003}.
\begin{proposition}\label{maxRWprop}
Let $\hat{Q}_{n}$ as in \eqref{g5a}. Then
\begin{equation*}
 \hat{Q}_{n}\Rightarrowd M_\beta,\qquad {\rm as}\ n\rightarrow\infty.
\end{equation*}
\end{proposition}

Although Proposition \ref{maxRWprop} tells us that the properly scaled $Q_{n}$ converges to a non-degenerate limiting random variable, it does not cover the convergence of its mean, variance and the empty-queue probability. In order to secure convergence of these performance measures as well, we follow the approach similar \cite{Sigman2011b}, using Assumptions \ref{as2} and \ref{as3}.

\begin{proposition}\label{prop6}
Let $\hat{Q}_{n}$ as in \eqref{g5a}, $\mu_n,\sigma_n^2 \rightarrow \infty$ such that both $\sigma_n^2/\mu_n\rightarrow \infty$ and $\E[\hat{A}_n^3]$ $<\infty$. Then
\begin{align*}
\label{b16}
\mathbb{P}(\hat{Q}_{n}= 0)&\rightarrow \mathbb{P}(M_\beta = 0),\\
\E [\hat{Q}_{n}]&\rightarrow \E [M_\beta],\\
\Var \hat{Q}_{n}&\rightarrow \Var M_\beta,
\end{align*}
as $n\rightarrow\infty$.
\end{proposition}
\proof
First, we recall that $\hat{Q}_{n}\equalD M_{\beta,n}$ for all $n\in\mathbb{N}$, so that $\mathbb{P}(\hat{Q}_{n} = 0) = \mathbb{P}(M_{\beta,n}=0)$, $\E[\hat{Q}_{n}]=\E[M_{\beta,n}]$ and $\Var\,\hat{Q}_{n}=\Var\,M_{\beta,n}$ as defined in \eqref{b1}. Our starting point is Spitzer's identity, see \cite[p.~230]{Asmussen2003},
\begin{equation}
\label{b17}
\E[\ee^{it M_{\beta,n}}] = \exp\Bigl( \sum_{k=1}^\infty \frac{1}{k} (\E[\ee^{itS_{k,n}^+}]-1)\Bigr),
\end{equation}
with $S_{k,n}$ as in \eqref{b1}, and $M_{\beta,n}$ the all-time maximum of the associated random walk. Simple manipulations of \eqref{b17} give
\begin{align}
\label{y1}
{\rm ln}\,\mathbb{P}(M_{\beta,n} = 0) &=  -\sum_{k=1}^\infty \frac{1}{k}\,\mathbb{P}(S_{k,n} > 0),\\
\label{y2}
\E[M_{\beta,n}] &= \sum_{k=1}^\infty \frac{1}{k} \E[S^+_{k,n}] = \sum_{k=1}^\infty \frac{1}{k}\int_0^\infty \mathbb{P}(S_{k,n} > x) \dd x,\\
\label{y3}
\Var M_{\beta,n} &= \sum_{k=1}^\infty \frac{1}{k} \E[(S^{+}_{k,n})^2] =\sum_{k=1}^\infty \frac{1}{k}\int_0^\infty \mathbb{P}(S_{k,n} > \sqrt{x}) \dd x.
\end{align}
By Lemma \ref{gaussStep}, we know
\begin{equation*}
\label{y4}
\mathbb{P}(S_{k,n} > y) = \mathbb{P}\left( {\sum_{i=1}^k} Y_{i,n} > y \right) \rightarrow \mathbb{P}\left(\sum_{i=1}^k Z_i> y\right),
\end{equation*}
for $n\rightarrow \infty$, where the $Z_i$'s are independent and identically normally distributed with mean $-\beta$ and variance 1.
Because equivalent expressions to \eqref{y1}-\eqref{y3} apply to the limiting Gaussian random walk, it is sufficient to show that the sums converge uniformly in $n$, so that we can apply dominated convergence to prove the result.

We start with the empty-queue probability. To justify interchangeability of the infinite sum and limit, note
\begin{equation*}
\label{y5}
 \mathbb{P}(S_{k,n} > 0) \leq  \mathbb{P}(|S_{k,n}+k\beta| > k\beta )\leq \frac{k}{\beta^2k^2} = \frac{1}{\beta^2k},
 \end{equation*}
where we used that $\E[ S_{k,n}] = k\E [Y_{1,n}] = -k\beta$ and $\Var S_{k,n} = k$ and apply Chebychev's inequality, so that
\begin{equation*}
\label{y6}
\sum_{k=1}^\infty \frac{1}{k}\mathbb{P}(S_{k,n} > 0) \leq \sum_{k=1}^\infty \frac{1}{\beta^2 k^2} < \infty, \qquad \forall n\in\mathbb{N}.
\end{equation*}
Hence,
\begin{align*}
\lim_{n\rightarrow\infty} {\rm ln}\,\mathbb{P}(\hat{Q}_{n}= 0) &= \lim_{n\rightarrow\infty}  - \sum_{k=1}^\infty \frac{1}{k}\mathbb{P}(S_{k,n} > 0) = -\sum_{k=1}^\infty \frac{1}{k} \lim_{n\rightarrow\infty}\mathbb{P}(S_{k,n} > 0)\nonumber\\
&= -\sum_{k=1}^\infty \frac{1}{k} \mathbb{P}({\sum_{i=1}^k} Z_i > 0) = {\rm ln}\,\mathbb{P}(M_\beta = 0),
\end{align*}
Finding a suitable upper bound on $\frac{1}{k}\int_0^\infty \mathbb{P}(\hat{Q}_{n}>x) dx$ and $\frac{1}{k}\int_0^\infty \mathbb{P}(\hat{Q}_{n}>\sqrt{x}) \dd x$ requires a bit more work. We initially focus on the former, the latter follows easily. The following inequality from \cite{Nagaev1979} proves to be very useful:
\begin{equation}
\label{y8}
\mathbb{P}(\bar{S}_k>y) \leq C_r\,\Bigl(\frac{k\,\sigma^2}{y^2}\Bigr)^2 + k\,\mathbb{P}(X>y/r),
\end{equation}
where $\bar{S}_k$ is the sum of $k$ i.i.d. random variables distributed as $X$, with $\E[X] = 0$ and $\Var\, X=\sigma^2$, $y > 0$, $r>0$ and $C_r$ a constant only depending on $r$. We take $r=3$ for brevity in the remainder of the proof, although any $r>2$ will suffice. We  analyze the integral in two parts, one for the interval $(0,k)$ and one for $[k,\infty)$. For the first part, we have
\begin{align}
\label{y9}
\int_0^k\mathbb{P}(S_{k,n}>x) \dd x &=\int_0^k \mathbb{P}({\sum_{i=1}^\infty}\hat{A}_{i,n} > x+k\beta)\dd x\, \leq\, \int_0^k \mathbb{P}({\sum_{i=1}^\infty}\hat{A}_{i,n} > k\beta)\dd x  \nonumber\\
&= k\,\mathbb{P}({\sum_{i=1}^k }\hat{A}_{i,n} >  k\beta) \,\leq\, \frac{C_3}{k^2\beta^6} + k^2\mathbb{P}(\hat{A}_{1,n}> \tfrac{1}{3}k),
\end{align}
where we used \eqref{y8} in the last inequality.
Hence,
\begin{align}
\label{y10}
\sum_{k=1}^\infty\frac{1}{k}\, \int_0^k \mathbb{P}(S_{k,n}>x)\dd x &\leq \, \frac{C_3}{\beta^6}\sum_{k=1}^\infty k^{-3} +\sum_{k=1}^\infty k\,\mathbb{P}(\hat{A}_{1,n}>\tfrac{1}{3}k) \nonumber \\
&\leq C_1^*+\sum_{k=1}^\infty k\,\mathbb{P}(\hat{A}_{1,n}>\tfrac{1}{3}k).
\end{align}
With the help of the inequality (see \cite{Sigman2011b}),
\begin{equation}
\label{y11}
(b-a)a\,\mathbb{P}(X>b) \leq \int_a^b x\,\mathbb{P}(X>x) \dd x \qquad \forall 0<a<b,
\end{equation}
we get by taking $a=(k-1)/3$ and $b=k/3$,
\begin{align}
\label{y12}
k\,\mathbb{P}(\hat{A}_{1,n}>\tfrac{1}{3}k) &\leq \frac{9\,k}{k-1}\int_{(k-1)/3}^{k/3} x\,\mathbb{P}(\hat{A}_{1,n}>x) \dd x \nonumber \\
&\leq 18\int_{(k-1)/3}^{k/3} x\,\mathbb{P}(\hat{A}_{1,n}>x) \dd x,
\end{align}
for $k\geq 2$. Since the tail probability for $k=1$ is obviously bounded by 1, this yields
\begin{align}
\label{y13}
\sum_{k=1}^\infty k\,\mathbb{P}(\hat{A}_{1,n}>\tfrac{1}{3}k) &\leq 1+18\sum_{k=2}^\infty\int_{(k-1)/3}^{k/3} x\,\mathbb{P}(\hat{A}_{1,n}>x) \dd x\nonumber\\
&\leq 1+ \int_{0}^{\infty} x\,\mathbb{P}(\hat{A}_{1,n}>x)\dd x  \leq 1+\E[\hat{A}_{1,n}^2] < \infty,
\end{align}
since $\hat{A}_{1,n}$ has finite variance by assumption. This completes the integral over the first interval. For the second part, we use \eqref{y8} again to find
\begin{align}
\label{y14}
\int_k^\infty \mathbb{P}(S_{k,n}>x)dx &=\int_k^\infty \mathbb{P}({ \sum_{i=1}^\infty}\hat{A}_{i,n} > x+k\beta)dx \leq \int_k^\infty \mathbb{P}({\sum_{i=1}^\infty}\hat{A}_{i,n} > x)\dd x\nonumber \\
&\leq C_3\int_k^\infty \frac{k^2}{x^6} dx + k\int_k^\infty \mathbb{P}(\hat{A}_{i,n}>\tfrac{1}{3}x)\dd x\nonumber \\
&=  \frac{5 C_3}{k^3}+ k\int_k^\infty \mathbb{P}(\hat{A}_{i,n}>\tfrac{1}{3}x) \dd x.
\end{align}
So,
\begin{equation}
\label{y15}
\sum_{k=1}^\infty \frac{1}{k} \int_k^\infty \mathbb{P}(S_{k,n}>x)dx \leq C_2^* + \sum_{k=1}^\infty \int_k^\infty \mathbb{P}(\hat{A}_{i,n}>\tfrac{1}{3}x) \dd x,
\end{equation}
for some constant $C_2^*$. Last, we are able to upper bound the second term in \eqref{y15} by Tonelli's theorem:
\begin{align}
\sum_{k=1}^\infty \int_k^\infty \mathbb{P}(\hat{A}_{i,n}>\tfrac{1}{3}x) \dd x &\leq \int_1^\infty x\mathbb{P}(\hat{A}_{1,n}>\tfrac{1}{3}x) \dd x \nonumber\\
\label{y16}
&\leq 9\int_0^\infty y\mathbb{P}(\hat{A}_{1,n}>y) \dd y = 9\E[\hat{A}_{1,n}^2] < \infty.
\end{align}
Combining the results in \eqref{y10}, \eqref{y13}, \eqref{y15} and \eqref{y16}, we find
\begin{equation*}
\label{y17}
\sum_{k=1}^\infty \frac{1}{k} \int_0^\infty \mathbb{P}(S_{k,n} >x)\dd x < \infty,
\end{equation*}
and thus
\begin{align*}
\lim_{n\rightarrow\infty} \E[\hat{Q}_{n}] &= \lim_{n\rightarrow\infty}  \sum_{k=1}^\infty \frac{1}{k}\int_0^\infty \mathbb{P}(S_{k,n} > x)\dd x \nonumber\\
&= \sum_{k=1}^\infty \frac{1}{k} \int_0^\infty\mathbb{P}({\sum_{i=1}^k} Z_i > x)\dd x = \E [M_\beta].
\end{align*}
Finally, we show how the proof changes for the convergence of $\Var \hat{Q}_{n}$. The expressions for $\E [\hat{Q}_{n}]$ and $\Var \hat{Q}_{n}$ in \eqref{y1} and \eqref{y2} only differ in the term $\sqrt{x}$. Hence only minor modifications are needed to also prove convergence of the variance. Note that boundedness of the integral over the interval $(0,k)$ in \eqref{y9}-\eqref{y13} remains to hold when substituting $\sqrt{x}$ for $x$. \eqref{y14} changes into
\begin{align*}
\label{y18}
\int_k^\infty \mathbb{P}(S_{k,n}>\sqrt{x})\dd x &=\int_k^\infty \mathbb{P}({\sum_{i=1}^\infty}\hat{A}_{i,n} > \sqrt{x}+k\beta)\dd x \nonumber \\
&\leq C_3\int_k^\infty \frac{k^2}{(\sqrt{x}+k\beta)^6} dx + k\,\int_k^\infty \mathbb{P}(\hat{A}_{1,n}>\tfrac{1}{3}\sqrt{x}) \dd x \nonumber\\
&\leq \frac{C_4^*}{k}+ k\,\int_k^\infty \mathbb{P}(\hat{A}_{1,n}>\tfrac{1}{3}\sqrt{x}) \dd x,
\end{align*}
for some constant $C_4^*$, so that
\begin{equation*}
\sum_{k=1}^\infty \frac{1}{k} \int_k^\infty \mathbb{P}(S_{k,n}>\sqrt{x})\dd x \leq C_4^* + \sum_{k=1}^\infty \int_k^\infty \mathbb{P}(\hat{A}_{1,n}>\tfrac{1}{3}\sqrt{x}) \dd x.
\end{equation*}
Lastly, we have
\begin{align*}
\sum_{k=1}^\infty \int_k^\infty \mathbb{P}(\hat{A}_{1,n}>\tfrac{1}{3}\sqrt{x}) \dd x &\leq \int_1^\infty x\mathbb{P}(\hat{A}_{1,n}>\tfrac{1}{3}\sqrt{x}) \dd x \nonumber\\
\label{y17a}
&\leq 18\int_0^\infty y^2\mathbb{P}(\hat{A}_{1,n}>y) \dd y = 18\E[\hat{A}_{1,n}^3] < \infty.
\end{align*}
Therefore the sum describing the variance is also uniformly convergent in $n$, so that interchanging of infinite sum and limit is permitted and
\begin{align*}
\lim_{n\rightarrow\infty} \Var\,\hat{Q}_{n} &= \lim_{n\rightarrow\infty}  \sum_{k=1}^\infty \frac{1}{k}\int_0^\infty \mathbb{P}(S_{k,n} > \sqrt{x})\dd x \nonumber \\
&= \sum_{k=1}^\infty \frac{1}{k} \int_0^\infty\mathbb{P}\Big({\sum_{i=1}^k} Z_i > \sqrt{x}\Big)\dd x = \Var M_\beta.
\end{align*}

\section{Numerical procedures}\label{numprocs}
An alternative characterization of the stationary distribution is based on the analysis in \cite{Boudreau1962} and considers a factorization in terms of (complex) roots:
\begin{equation*}
\label{t9}
Q_{n}(w) = \frac{(s_n-\E [A_{n}])(w-1)}{w^{s_n}-\tilde{A}_{n}(w)}\,\prod_{k=1}^{s_n-1} \frac{w-z^n_k}{1-z^n_k},
\end{equation*}
where $z_1^n,z_2^n...,z_{s_n-1}^n$ are the $s_n-1$ zeros of $z^{s_n}-\tilde{A}_{n}(z)$, in $|z|<1$, yielding
\begin{equation*}
\label{c2}
\E Q_n  = \frac{\sigma_n^2}{2(s_n-\mu_n)}-\frac{s_n-1+\mu_n}{2} + \sum_{k-1}^{s_n-1} \frac{1}{1-z^n_k},
\end{equation*}
\begin{equation*}
\label{c3}
\mathbb{P}(Q_{n}=0) = \frac{s_n-\mu_A}{\tilde{A}_{n}(0)}\prod_{k=1}^{s-1}\frac{z^n_k}{z^n_k-1},
\end{equation*}
which for our choice of $\tilde{A}_{n}(z)$ becomes
\begin{equation*}
\label{c4}
\E Q_n = \frac{a_nb_n(b_n+1)}{2\beta\sqrt{a_n}b_n}-\frac{2a_nb_n+\beta\sqrt{a_nb_n(b_n+1)}-1}{2}+\sum_{k=1}^{s_n-1} \frac{1}{1-z^n_k},
\end{equation*}
\begin{equation*}
\label{c5}
\mathbb{P}(Q_{n}=0) = \beta \sqrt{a_nb_n(b_n+1)}(1+b_n)^{a_n}\prod_{k=1}^{s_n-1} \frac{z^n_k}{z^n_k-1}.
\end{equation*}
where $z_1,...,z_{s_n-1}$ denote the zeros of $z^{s_n} - \tilde{A}_{n}(z)$ in $|z|<1$. These zeros exist under the assumption $s_n > a_nb_n$; see \cite{rouche}. A robust numerical procedure to obtain these zeros is essential for a base of comparison. We  discuss two methods that fit these requirements. The first follows directly from \cite{Janssen2005}. \\
\begin{lemma}\label{fixedIterLemma}
Define the iteration scheme
\begin{equation}
\label{c6}
z_k^{n,l+1} = w^n_k [\tilde{A}_{n}(z_k^{n,l})]^{1/s_n},
\end{equation}
with $w^n_k = \ee^{2\pi ik/s_n}$ and $z_k^{n,0}=0$ for $k=0,1,\ldots,s_{n-1}$. Then $z_k^{n,l} \rightarrow z_k^n$ for all $k=0,1,...,s_n-1$ for $l\rightarrow \infty$.
\end{lemma}

\begin{proof}
The successive substitution scheme given in \eqref{c6} is the fixed point iteration scheme described in \cite{Janssen2005}, applied to the pgf of our arrival process. The authors show that, under the assumption of $\tilde{A}_{n}(z)$ being zero-free within $|z|\leq 1$, the zeros can be approximated arbitrarily closely, given that the function $[\tilde{A}_{n}(z)]^{1/s_n}$ is a contraction for $|z|\leq 1$, i.e.
\begin{equation*}
\label{c7}
\Bigl|\frac{\dd}{\dd z}[\tilde{A}_{n}(z)]^{1/s_n}\Bigr| < 1.
\end{equation*}
In our case,
\begin{align}
\label{c8}
\Bigl|\frac{\dd}{\dd z}[\tilde{A}_{n}(z)]^{1/s_n}\Bigr| = \Bigl|\frac{\dd}{\dd z}\left(1+(1-z)b_n\right)^{-a_n/s_n}\Bigr| = \frac{a_nb_n}{s_n}\Bigl|1+(1-z)b_n\Bigr|^{-a_n/s_n-1},
\end{align}
where $a_nb_n/s_n = \rho_n$ is close to, but less than 1 and
\begin{align*}
\label{c9}
|1+(1-z)b_n| \geq |1+b_n|-|z|b_n = 1+(1-|z|)b_n \geq 1,
\end{align*}
when $|z|\leq 1$. Hence the expression in \eqref{c8} is less than 1 for all $|z|\leq 1$. Evidently, $\tilde{A}_{n}(z)$ is also zero-free in $|z|\leq 1$. Thus \cite[Lemma~3.8]{Janssen2005} shows that $z_k^{n,l}$ as in \eqref{c6} converges to the desired roots $z^n_k$ for all $k$ as $l$ tends to infinity.
\end{proof}

\begin{remark}
The asymptotic convergence rate of the iteration in \eqref{c6} equals \\
\noindent $\frac{\dd}{\dd z}[\tilde{A}_{n}(z)]^{1/s_n}$ evaluated at $z=z_k^n$. Hence, convergence is slow for zeros near 1 and fast for zeros away from 1.
\end{remark}

A different approach is based on the B\"urmann-Lagrange inversion formula.
\begin{lemma}\label{BLLemma}
Let $w^n_k = e^{2\pi ik/s_n}$ and $\alpha_n = a_n/s_n$. Then the zeros of $z^{s_n}-\tilde{A}_{n}(z)$ are given by
\begin{equation*}
z_k^n = \sum_{l=1}^\infty \frac{1}{l!}\,\frac{\beta[l\alpha_n+l-1)}{\beta(l\alpha_n)}\,\frac{b_n+1}{b_n}\Bigl(\frac{b_n}{(b_n+1)^{\alpha_n+1}}\Bigr)^l (w_k^n)^l,
\end{equation*}
for $k=0,1,...,s_n-1$.
\end{lemma}

\begin{proof}
Note that we are looking for $z$'s that solve
\begin{equation*}
\label{c10}
z\,[\tilde{A}_{n}(z)]^{-1/s_n} = z\left(1+(1-z)b_n\right)^{a_n/s_n} = w,
\end{equation*}
where $w = w_k = \ee^{2\pi i k/s_n}$. The B\"urmann-Lagrange formula for $z=z(w)$, as can be found in \cite[Sec.~2.2]{debruijn} for $z=z(w)$ is given by
\begin{align*}
z(w) &= \sum_{l=1}^\infty \frac{1}{l!}\,\left(\frac{\dd}{\dd z}\right)^{l-1}\left[\left(\frac{z}{z(1+(1-z)b_n)^{a_n/s_n}}\right)^l\right]_{z=0}\,w^l\nonumber\\
\label{c11}
&= \sum_{l=1}^\infty \frac{1}{l!}\,\left(\frac{\dd}{\dd z}\right)^{l-1}\left[\left(1+(1-z)b_n)^{-l\,a_n/s_n}\right)\right]_{z=0}\,w^l.
\end{align*}
Set $\alpha_n = a_n/s_n$. We compute
\begin{equation*}
\label{c1}
\left(\frac{\dd}{\dd z}\right)^{l-1}\left[ (1+(1-z)b_n)^{-l\alpha_n}\right]_{z=0} = \frac{\beta(l\alpha_n+l-1)}{\beta(l\alpha_n)}\,\frac{1+b_n}{b_n}\,\left(\frac{b_n}{(1+b_n)^{\alpha_n+1}}\right)^l.
\end{equation*}
With $c_n = b_n/(1+b_n)^{\alpha_n+1}$ and $d_n = (1+b_n)/b_n$, we thus have
\begin{equation*}
\label{c13}
z(w) = d_n\,\sum_{l=1}^\infty \frac{\beta(l\alpha_n+l-1)}{\beta(l+1)\beta(l\alpha_n)} c_n^l\,w^l.
\end{equation*}
By Stirling's formula
\begin{equation*}\label{c14}
\frac{\beta(l\alpha_n+l-1)}{\beta(l+1)\beta(l\alpha_n)}  = \frac{D}{l\sqrt{l}}\left(\frac{(\alpha_n+1)^{\alpha_n+1}}{\alpha_n^{\alpha_n}}\right)^l,
\end{equation*}
where $D=\alpha_n^{1/2}(\alpha_n+1)^{-3/2}(2\pi)^{-1/2}$. Now,
\begin{equation*}
\label{c15}
\frac{(\alpha_n+1)^{\alpha_n+1}}{\alpha_n^{\alpha_n}}\, c_n = \frac{(\alpha_n+1)^{\alpha_n+1}}{\alpha_n^{\alpha_n}}\cdot \frac{b_n}{(1+b_n)^{\alpha_n+1}} = \left(\frac{b_n+\rho_n}{b_n+1}\right)^{\rho_n/b_n + 1}\left(\frac{1}{\rho_n}\right)^{\rho_n/b_n}.
\end{equation*}
This determines the radius of convergence $r_{\rm BL}$ of the above series for $z(w)$:
\begin{equation}
\label{c16}
\frac{1}{r_{\rm BL}} := \left(\frac{b_n+\rho_n}{b_n+1}\right)^{\rho_n/b_n + 1}\left(\frac{1}{\rho_n}\right)^{\rho_n/b_n}.
\end{equation}
The derivative with respect to $\rho_n$ of the quantity
\begin{equation}
\label{c17}
\left(1+\frac{\rho_n}{b_n}\right) {\rm ln }\left(\frac{b_n+\rho_n}{b_n+1}\right)+\frac{\rho_n}{b_n}\,{\rm ln}\left(\frac{1}{\rho_n}\right)
\end{equation}
is given by
\begin{equation*}
\label{c18}
\frac{1}{b_n}{\rm ln }\Bigl(\frac{b_n+\rho_n}{b_n\rho_n+\rho_n}\Bigr) > 0
\end{equation*}
for $0<\rho_n<1$ and $b_n>0$. Furthermore, the quantity in \eqref{c17} vanishes at $\rho_n=1$ and is therefore negative for $0<\rho_n<1$ and $b_n>0$.
\begin{remark}
The formula for the radius of convergence in \eqref{c16} clearly shows the decremental effect of both having a large $b_n$ and or having $\rho_n$ close to 1. The quantities $\beta(l\alpha+l-1)/(\beta(l+1)\beta(l\alpha))$ in the power series for $z(w)$ are not very convenient for recursive computation, although normally $\alpha = a_n/s_n$ is a rational number.\end{remark}
\end{proof}

\end{document}